\documentclass[12pt]{article}

\usepackage{hyperref,amsmath,amssymb,amscd,amsthm,makeidx,txfonts,graphicx,url,bm}
\usepackage{a4wide,xy}
\usepackage{float}

\numberwithin{equation}{subsection}

\let\oldmarginpar\marginpar
\renewcommand\marginpar[1]{\-\oldmarginpar[\raggedleft\footnotesize #1]
{\raggedright\footnotesize #1}}

\makeindex

\xyoption{all}
\input{xypic}

\newtheorem{theorem}{Theorem}[subsection]
\newtheorem{proposition}[theorem]{Proposition}
\newtheorem{corollary}[theorem]{Corollary}

\newtheorem{question}[theorem]{Question}

\theoremstyle{remark}
\newtheorem{remark}[theorem]{Remark}
\newtheorem{example}[theorem]{Example}

\theoremstyle{definition}
\newtheorem{definition}[theorem]{Definition}

\newcounter{margin}
{\end{itshape}  \bigskip}

\def\bes{\begin{eqnarray*}}
\def\ees{\end{eqnarray*}}

\def\vv{{\bf v}}

\DeclareMathOperator{\Aut}{Aut} 
 \DeclareMathOperator{\Hom}{Hom}

\DeclareMathOperator{\Hilb}{Hilb}

\DeclareRobustCommand{\stirling}{\genfrac\{\}{0pt}{}}

\def\A{{\bf A}}

\def\V{\mathbb{V}}
\def\C{\mathbb{C}}
\def\M{{\mathcal{M}}}

\def\calV{{\mathcal{V}}}
\def\calE{{\mathcal{E}}}

\def\calP{\mathcal{P}}
\def\calH{\mathcal{H}}

\def\1{{\bf 1}}

\def\v{\mathbf{v}}
\def\w{\mathbf{w}}
\def\P{\mathcal{P}}
\def\pihat{{\bf \pi}}

\def\N{\mathbb{Z}_{\geq 0}}

\def\R{\mathbb{R}}
\def\F{\mathbb{F}}

\def\Q{\mathbb{Q}}

\def\calC{{\mathcal C}}

\def\Z{\mathbb{Z}}

\def\K{\mathbb{K}}

\def\gl{{\mathfrak g\mathfrak l}}

\newcommand{\nc}{\newcommand}
\nc{\bM}{{\mathbb M}}
\nc{\Sp}{{\rm Sp}}

\nc{\op}[1]{\mathop{\mathchoice{\mbox{\rm #1}}{\mbox{\rm #1}}
{\mbox{\rm \scriptsize #1}}{\mbox{\rm \tiny #1}}}\nolimits}
\nc{\al}{\alpha}

\nc{\ep}{\varepsilon} \nc{\ga}{\gamma} \nc{\Ga}{\Gamma}
\nc{\la}{\lambda} \nc{\La}{\Lambda} \nc{\si}{\sigma}
\nc{\Sig}{{\Gamma}} \nc{\Om}{\Omega} \nc{\om}{\omega}

\nc{\SL}{{\rm SL}} \nc{\GL}{{\rm GL}} \nc{\PGL}{{\rm PGL}}
\nc{\G}{{\rm G}}
\nc{\bV}{{\mathbb V}}
\nc{\g}{{\mathfrak g}}
\nc{\Gr}{{\rm Gr}}

\nc{\beq}[1]{\begin{eqnarray}\label{#1}}
\nc{\eeq}{\end{eqnarray}}

\renewcommand{\P}{{\mathbb P}}

\nc{\cpt}{{\op{cpt}}} \nc{\Dol}{{\op{Dol}}} \nc{\DR}{{\op{DR}}}
\nc{\B}{{\op{B}}} \nc{\Triv}{\op{Triv}} \nc{\Hod}{{\op{Hod}}}
\nc{\Log}{{\op{Log}}} \nc{\Exp}{{\op{Exp}}} \nc{\Est}{E_{\op{st}}}
\nc{\Hst}{H_{\op{st}}} \nc{\Left}[1]{\hbox{$\left#1\vbox to
  10.5pt{}\right.\nulldelimiterspace=0pt \mathsurround=0pt$}}
\nc{\Right}[1]{\hbox{$\left.\vbox to
  10.5pt{}\right#1\nulldelimiterspace=0pt \mathsurround=0pt$}}
\nc{\LEFT}[1]{\hbox{$\left#1\vbox to
  15.5pt{}\right.\nulldelimiterspace=0pt \mathsurround=0pt$}}
\nc{\RIGHT}[1]{\hbox{$\left.\vbox to
  15.5pt{}\right#1\nulldelimiterspace=0pt \mathsurround=0pt$}}

\nc{\bee}{{\bf E}} \nc{\bphi}{{\bf \Phi}}

\begin{document}

\title{Cohomology of large semiprojective hyperk\"ahler varieties}

\author{ Tam\'as Hausel
\\ {\it EPF Lausanne}
\\{\tt tamas.hausel@epfl.ch}
 \and Fernando Rodriguez-Villegas
\\ 
{\it  ICTP Trieste} \\ {\tt villegas@ictp.it}\\  }

\pagestyle{myheadings}

\maketitle

{\begin{center}{\it \`a G\'erard Laumon \`a l'occasion de son
60\'eme anniversaire}\end{center}}

\begin{abstract} In this paper we survey geometric and arithmetic techniques to study the
  cohomology of semiprojective hyperk\"ahler manifolds including toric
  hyperk\"ahler varieties, Nakajima quiver varieties and moduli spaces
  of Higgs bundles on Riemann surfaces. The resulting formulae for
  their Poincar\'e polynomials are combinatorial and representation
  theoretical in nature.  In particular we will look at their Betti
  numbers and will establish some results and expectations on their
  asymptotic shape.
\end{abstract}

At the conference ``De la g\'eom\'etrie alg\'ebrique aux formes
automorphes : une conf\'erence en l'honneur de G\'erard Laumon" the
first author gave a talk, whose subject is well-documented in the
survey paper \cite{hausel-survey}. Here, instead, we will discuss
techniques, both geometrical and arithmetic, for obtaining information
on the cohomology of semiprojective hyperk\"ahler varieties and we
will report on some observations on the asymptotic behaviour of their Betti
numbers in certain family of examples.

We call $X$ a smooth quasi-projective variety with a
$\C^\times$-action {\em semiprojective} when the fixed point set
$X^{\C^\times}$ is projective and for every $x\in X$ and as
$\lambda\in \C^\times$ tends to $0$ the limit $\lim_{\lambda \to 0}
\lambda x$ exists.

Varieties with these assumptions were originally studied by Simpson in
\cite[\S 11]{simpson-nonabelian} and varieties with similar
assumptions were studied by Nakajima in \cite[\S
5.1]{nakajima-hilbert}. The terminology semiprojective in this context
appeared in \cite{hausel-sturmfels}, which concerned semiprojective
toric varieties and toric hyperk\"ahler varieties. In particular, a
large class of hyperk\"ahler varieties, which arise as a hyperk\"ahler
quotient of a vector space by a gauge group, are semiprojective.
These include Hilbert schemes of $n$-points on $\C^2$, Nakajima quiver
varieties and moduli spaces of Higgs bundles on Riemann surfaces.

It turns out that despite their simple definition
 we can say quite a lot about the geometry and
cohomology of semiprojective varieties.  We can construct a
Bialinycki-Birula stratification (\S \ref{bb}), which in
\S\ref{csp} will give a perfect Morse stratification in the sense of
Atiyah--Bott. This way we will be able to
deduce that the cohomology of a semiprojective variety is isomorphic
with the cohomology of the fixed point set $X^{\C^\times}$ with some
cohomological shifts. Also, the opposite Bialinycki-Birula
stratification will stratify a projective subvariety $\calC\subset X$
of the semiprojective variety, the so-called {\em core}, which turns out to be a deformation
retract of $X$. This way we
can deduce that the cohomology $H^*(X;\C)$ is always
pure. Furthermore, we can compactify $\overline X=X\coprod Z$ with a
divisor $Z$, to get an orbifold $\overline{X}$. Finally in \S\ref{whl} we will
look at a version of a weak form of the Hard Lefschetz theorem
satisfied by semiprojective varieties.

We will also discuss arithmetic approach to obtain information on the
cohomology of our hyperk\"ahler varieties. It turns out that the
algebraic symplectic quotient construction of our hyperk\"ahler
varieties will enable us to use a technique we call {\em arithmetic
  harmonic analysis} to count the points of our hyperk\"ahler
varieties over finite fields. With this technique we can effectively
determine the Betti numbers of the toric hyperk\"ahler varieties and
Nakajima quiver varieties as well as formulate a conjectural
expression for the Betti numbers of the moduli space of Higgs bundles.

To test the range in which the Weak Hard Lefschetz theorem of
\S\ref{whl} might hold, we will look at the graph of Betti numbers for
our varieties when their dimension is very large.  The resulting
pictures are fairly similar and we observe that asymptotically they
seem to converge to the graph of some continuous functions.  We will
see, for example, the normal, Gumbel and Airy distributions emerging
in the limit in our examples. We will conclude the paper with some
proofs and heuristics towards establishing such facts.

\begin{paragraph}{\bf Acknowledgement} We would like to thank G\'abor Elek, Stavros
  Garoufadilis, Sergei Gukov, Jochen Heinloth, Daniel Huybrechts,
  Andrew Morrison, Antonello Scardicchio, Christoph Sorger,  Bal\'azs
  Szegedy and Bal\'azs Szendr\H{o}i for 
  discussions related to this paper. The first author was supported by
  a Royal Society University Research Fellowship  (2005-2012)
  and  by the  Advanced Grant "Arithmetic and physics of Higgs moduli spaces" no. 320593 of the European Research Council (2013-2018) during work on this paper. 
 The second author  is supported by the NSF grant
DMS-1101484 and a Research Scholarship from the Clay Mathematical
Institute. He would also like to thank the Mathematical Institute of
University of Oxford where this work was started for its hospitality.
\end{paragraph}

\section{Semiprojective varieties}

\subsection{Definition and examples}
\label{semiproj}

We start with the definition of a semiprojective variety, first considered in \cite[Theorem 11.2]{simpson-nonabelian}.

\begin{definition}\label{semi} Let $X$ be a complex quasi-projective algebraic variety with a 
$\C^\times$-action. We call $X$ {\em semiprojective} when the following
two assumptions hold: 
\begin{enumerate}\item The fixed point set $X^{\C^\times}$ is proper.
\item For every $x\in X$ the  
$\lim_{\lambda \to 0} \lambda x$  exists  as $\lambda\in \C^\times$ tends to $0$.
\end{enumerate}
\end{definition}

The second condition could be phrased more algebraically as follows: for every $x\in X$ we have an equivariant map $f:\C\to X$ such that $f(1)=x$ and $\C^\times$ acts on $\C$ by multiplication. 

First  example is a projective variety with a trivial (or  any) 
$\C^\times$-action. 
For a large class of non-projective examples one can take the total space of a vector bundle on a projective variety, which together with the canonical $\C^\times$-action will become semiprojective. 

%An affine variety $X$ can only be semiprojective, when the projective fixed point set is $0$ dimensional. Additionally when the variety is connected it could only have a single fixed point. It follows
%that the affine semiprojective varieties are weighted cones over projective varieties. 

A good source of examples arise by taking GIT quotients of linear group
actions of reductive groups on vector spaces. Examples include the semiprojective toric varieties of \cite{hausel-sturmfels} (even though the definition of semiprojectiveness is different there, but
 equivalent with ours) and quiver varieties studied by Reineke \cite{reineke1}.

\subsubsection{Semiprojective hyperk\"ahler varieties}
\label{hksp}
In this survey we are interested in semiprojective hyperk\"ahler
varieties. Examples arise by taking the algebraic symplectic
quotient of a complex symplectic vector space $\bM$ by a symplectic
linear action of a reductive group $\overline{\rho}:\G\to
\Sp(\bM)$. In practice $\bM=\bV\times \bV^*$ and $\overline{\rho}$
arises as the doubling of a representation $\rho:\G\to \GL (\bV)$.  If
$\g$ denotes the Lie algebra of $\G$, we have the derivative of $\rho$
as $\varrho:\g\to\gl(\bV)$. This gives us the moment map
\begin{equation}
\label{momentmap}\mu: \bM\to \g^*,\end{equation}  at $x\in \g$ by the formula \begin{equation}\label{momentmapdef} 
\langle \mu(v,w),x\rangle=\langle \varrho(x)v,w\rangle.
\end{equation}  By construction $\mu$ is equivariant with respect to the coadjoint action of $\G$ on $\g^*$.
Taking a character $\sigma\in \Hom(\G,\C^\times)$ will yield the GIT quotient
$\M_\sigma:=\mu^{-1}(0)/\!/_\sigma\G$ using the linearization induced by $\sigma$. Sometimes 
$\sigma$ can be chosen generically so that $\M_\sigma$ becomes non-singular (and by construction) quasi-projective. We assume this henceforth. By construction of the GIT quotient we have the proper
affinization map \begin{equation}\label{affinization} \M_\sigma\to\M_0 \end{equation} to the affine GIT quotient $\M_0=\mu^{-1}(0)/\!/\G.$ 

The $\C^\times$-action on $\bM$ given by dilation will
commute with the linear action of $\G$ on it so that the moment map 
\eqref{momentmap} will be equivariant with respect to  this and the weight $2$ action of $\C^\times$ on $\g^*$. This will induce a $\C^\times$-action on $\M_\sigma$, such that on the
affine GIT quotient $\M_0$ it will have a single fixed point 
corresponding to the origin in $\mu^{-1}(0)\subset \bM$. This and the fact
that the affinization map \eqref{affinization} is proper implies, that $\M_\rho^\sigma$ is 
semiprojective, provided that $\M_\rho^\sigma$ is non-singular, which we always assume. 

An important special case is when $Z(\GL(\V))\subset {\rm im}
\rho$. In this case we can take a square root of the $\C^\times$
action above by acting only on $\V^\times$ by dilation and trivially
on $\V$. This action will also commute with the action of $\G$ on
$\bM=\V\times \V^\times$ and will indeed reduce to a
$\C^\times$-action on the quotient $\M_\rho^\sigma$ whose square is
the $\C^\times$-action we considered in the previous paragraph. In
particular this new action also makes $\M_\rho^\sigma$ semiprojective. In fact, it will have an additional property. Namely, the
natural symplectic form $\omega_\bM$ on $\bM$ will be of homogeneity
$1$ with respect to the $\C^\times$ action; in other words, it will
satisfy \begin{equation}\label{hyper} \lambda^*(\omega_\bM)=\lambda
  \omega_\bM \end{equation} under this action. This property will be
inherited by the quotient $\M_\rho^\sigma$. Following
\cite{hausel-qgm} we make the following

\begin{definition}
\label{hyper-compact}
A semiprojective hyperk\"ahler variety with a symplectic form  of  homogeneity one as in
\eqref{hyper} is called {\em hyper-compact}.
\end{definition}

When $\G$ is a torus, $\M_\rho^\sigma$ are the toric hyperk\"ahler
varieties of \cite{hausel-sturmfels}; these always can be arranged to become
hyper-compact. When the representation $\rho$
arises from a quiver with a dimension vector $\M_\rho^\sigma$ is a
quiver variety as constructed by Nakajima in \cite{nakajima}. When the quiver has no
edge loops, one can always arrange that $\M_\rho^\sigma$ becomes hyper-compact.
When the quiver is the tennis-racquet quiver, i.e. two
vertices connected with a single edge and with a loop on one of them,
and the dimension vector is $1$ in the simple vertex and $n$ on the
looped one, the Nakajima quiver variety becomes isomorphic with
$(\C^2)^{[n]}$ the Hilbert scheme of $n$ points on $\C^2$. This semiprojective
hyperk\"ahler variety however
 is not hyper-compact as we will see later.

Finally, the following hyper-compact examples originally arose from an infinite dimensional
analogue of the above construction. In \cite{hitchin} Hitchin
constructs the moduli space of semistable rank $n$ degree $d$ Higgs
bundles on a Riemann surface as an infinite dimensional gauge
theoretical quotient. A Higgs bundle is a pair $(E,\phi)$ of a rank
$n$ degree $d$ vector bundle $E$ on the Riemann surface $C$ and
$\phi\in H^0(C;{\rm End}(E)\otimes K_C)$. Nitsure \cite{nitsure}
constructed such moduli spaces $\M^d_n$ in the algebraic geometric
category, which are non-singular quasi-projective varieties when
$(n,d)=1$. There is a natural $\C^\times$ action on $\M^d_n$ given by
scaling the Higgs field $(E,\phi)\mapsto (E,\lambda \phi)$. Hitchin
\cite{hitchin} when $n=2$ and Simpson \cite[Corollary
10.3]{simpson-nonabelian} in general showed that $\M^d_n$ is
semiprojective. A nice argument to see this, is similar for the
argument for $\M_\sigma$ above. Namely the
affinization \begin{equation}\label{hitchinmap} \chi: \M^d_n\to
  \A \end{equation} turns out to be the famous Hitchin map
\cite{hitchin-integrable}, which by results of Hitchin \cite{hitchin}
when $n=2$ and Nitsure \cite{nitsure} for general $n$ is a proper
map. It is also $\C^\times$-equivariant wich covers a $\C^\times$-action on the affine $\A$ with a single fixed point.  This implies
that $\M^d_n$ is indeed semiprojective.

\subsection{Bialinycki-Birula decomposition of semiprojective
  varieties}
\label{bb}
Much in this section is due to Simpson \cite{simpson-nonabelian}, Nakajima \cite{nakajima} and Atiyah--Bott \cite{atiyah-bott}. 

Let $X$ be a non-singular semiprojective variety. 
Let $X^{\C^\times}=\coprod_{i\in I} F_i$
be the decomposition of the fixed point set into connected components. Then $I$ is finite and $F_i$ are non-singular projective subvarieties of $X$. According to \cite[Corollary 7.2]{dolgachev} we can linearize the action of $\C^\times$ on a very ample line bundle $L$ on $X$. On each $F_i$ then $\C^\times$   will act on $L$  through
a homomorphism $\C^\times \to \C^\times$ with weight $\alpha_i\in \Z$ which we can assume, by suitably changing the linearisation, are always non-negative $\alpha_i\in \N$. We introduce a partial ordering on $I$ by setting 
\begin{equation}\label{order}i<j \qquad \Leftrightarrow \qquad
  \alpha_i>\alpha_j.\end{equation}

Introduce $U_i\subset X$ as the set of points $x\in X$ for which 
$\lim_{t\rightarrow 0}tx\in F_i$. 
Similarly, as above, we can define
$D_i$ as the points $x\in X$ for which 
$\lim_{t\rightarrow \infty}\lambda x\in 
F_i$. These are locally closed subsets and Bialynicki-Birula
\cite[Theorem 4.1] {bialynicki} proves that both $U_i$ and $D_i$ are subschemes
of $X$ which are isomorphic to  
certain affine
 bundles (so-called $\C^\times$-fibrations) over $F_i$. 

 It will be convenient to make the following
\begin{definition}
\label{core-defn}
The {\em core} of the semiprojective variety $X$ is
$$
\calC:=\cup_{i\in I} D_i \subset X
$$ 
\end{definition}

 By assumption 2 of Definition~\ref{semi} we get the Bialinycki-Birula decomposition 
$X=\coprod_{i\in I} U_i$.  This decomposition  satisfies that \begin{equation}\label{closure} \overline{U_i}\subset \cup_{j\geq i} U_i.\end{equation} To see this we note that using the linearisation on the very ample line bundle $L$ we can equivariantly embed $X$ into some projective space $\P^N$ with a linear $\C^\times$ action. \eqref{closure} follows from the corresponding statement for the linear action of $\C^\times$ on $\P^N$, where it is clear.

It follows from the Hilbert-Mumford criterium for semistability that
$X^{ss}=X\setminus \coprod_{i\in I} D_i$ with respect to our
linearisation. Thus we have a geometric quotient
$Z:=X^{ss}/\C^\times$, which is proper according to \cite[Theorem
11.2] {simpson-nonabelian} and is, in fact, an orbifold as there are
no fixed points of $\C^\times$ on $X^{ss}$. Using this construction
for the semiprojective $X\times \C$ where $\C^\times$ acts via the
diagonal action (with the standard multiplication action on the second
factor) we get
 \begin{equation}
\label{compact}
\overline{X}:=(X\times
\C)^{ss}/\C^\times,
\end{equation}
which decomposes as $\overline{X}=X\coprod Z$ corresponding to points
in $(X\times\C)^{ss}$ with non-zero (respectively zero) second
component.  This thus yields an orbifold compactification of $X$, the
algebraic analogue of Lerman's symplectic cutting \cite{lerman}, as
studied in \cite{hausel-comp}.

An immediate consequence of this compactification is the following:
\begin{corollary}
\label{proper}
The core $\calC$ of a semiprojective variety $X$ is proper.
\end{corollary}
\begin{proof} The proper $\overline{X}$ has two Bialinycki-Birula
  decompositions. One of them is $$\overline{X}=D_\infty\cup
  \coprod_{i\in I} D_i$$ where $$D_\infty=(X\setminus \calC)\cup
  Z\subset \overline{X}.$$ Thus by property \eqref{closure} the core
  $\calC=\coprod_{i\in I} D_i$ is closed in the proper
  $\overline{X}$. The claim follows.
\end{proof}

\subsection{Cohomology of semiprojective varieties}
\label{csp}
\subsubsection{Generalities on cohomologies of complex algebraic varieties}
We denote by $H^*(X;\Z)$ the integer and by $H^*(X;\Q)$ the rational singular cohomology of a CW complex $X$.
$H^*(X;\Z)$ is a graded anti commutative ring; while $H^*(X;\Q)$ is a graded anticommutative $\Q$-algebra. 

When $X$ is a complex algebraic variety there is further structure on its rational cohomology. 
 Motivated by the Frobenius action on the $l$-adic cohomology
of a variety defined over an algebraic closure of a finite field Deligne in 1971 \cite{deligne} 
introduced mixed Hodge structures on the cohomology of any complex algebraic variety $X$. 

Here we only recall the notion of the weight filtration on rational
cohomology. It is an increasing
filtration: $$W_0(H^k(X;\Q))=0\subset\dots\subset
W_{i}(H^k(X;\Q))\subset\dots \subset W_{2k}(H^k(X;\Q))=H^k(X;\Q)$$ by
$\Q$-vector spaces $W_{i}(H^k(X;\Q))$. It has many nice
properties. For example it is functorial, 
\begin{equation}\label{smooth}
  W_{k-1}(H^{k}(X;\Q))=0\end{equation} for a smooth $X$,
and \begin{equation}\label{projective}W_k(H^k(X;\Q))=H^k(X;\Q)\end{equation}
for a projective variety $X$. We say that the weight filtration on
$H^*(X;\Q)$ is {\em pure} when both \eqref{smooth} and
\eqref{projective} holds for every $k$. In particular a smooth
projective variety always has {pure} weight filtration. We will see in
Corollary~\ref{pure} that semiprojective varieties also have pure
weight filtration.

We denote by  $$H(X;q,t)=\sum_{i,k} \dim(\Gr^W_iH^k(X;\Q)) q^{i/2}t^k\in \Z[q^{1/2},t]$$ the {\em mixed Hodge polynomial}. It has two important specializations. The polynomial  
$$P(X;t)=H(1,t)= \sum_k  \dim(H^k(X;\Q))t^k\in \Z[t]$$
is the {\em Poincar\'e polynomial of $X$}, while the specialization
\begin{equation}\label{epolydef}E(X;q)=q^{\dim X} H( X;1/q,-1)=\sum_{i,k} (-1)^k \dim(\Gr^W_iH^k(X;\Q)) q^{{\dim X}-i/2}\in \Z[q^{1/2}] \end{equation} the
{\em virtual weight polynomial}. In the case when the weight filtration is pure on $H^*(X;\Q)$ we have the relationships
\begin{equation}\label{purel}H(X;q,t)=P(X;q^{1/2}t)=E(X;qt^2).
\end{equation} In the general case however there is no such relationships.

\subsubsection{The case of semiprojective varieties}
Let $X=\coprod_{i\in I} U_i$ the Bialinycki-Birula decomposition of a
semiprojective variety, with index set $I$ given a partial ordering as
in \eqref{order} . Following \cite[pp 537]{atiyah-bott}let $J\subset
I$ such that \begin{equation}\label{open}j\in I \mbox{ and } i<j
  \mbox{ implies } i\in J.\end{equation} Then by $U_J:=\cup_{j\in J}
U_j$ is open in $X$ by \eqref{closure}. Let $\lambda\in I\setminus J$
be minimal and let $J^+:=J\cup \lambda$, this also satisfies
\eqref{open} so $U_{J^+}$ is also open in $X$ and $U_{\lambda}$ is
closed in $U_{J^+}$. Furthermore the open subvarieties $U_{J}$ and
$U_{J^+}$ of $X$ are both semiprojective with core
$$D_J:=\coprod_{j\in J} D_j\subset U_J$$ and $$D_{J^+}=\coprod_{j\in J^+} D_j=D_J\cup D_\lambda\subset U_{J^+}.$$ 
We now have the following commutative diagram:
 \begin{equation}\label{compare}\begin{array}{ccccccc}  \to & H^{j-k_\lambda}(U_{\lambda};\Z)& \to & H^{j}(U_{J^+};\Z)&\to& H^{j}(U_J;\Z)&\to \\ &\downarrow i^*_\lambda &&\downarrow i_J^* &&\downarrow i_{J^+}^* & \\  \to & H^{j-k_\lambda}(F_{\lambda};\Z)& \to & H^{j}(D_{J^+};\Z)&\to& H^{j}(D_J;\Z)&\to  \end{array}. \end{equation} Here the top row  is the cohomology
long-exact sequence of the pair $(U_{J^+},U_J)$ and $$H^{j-k_\lambda}(U_{\lambda},\Z)\cong H^j(U_{J^+},U_J;\Z)$$ is excision followed by the Thom isomorphism theorem, where $k_\lambda={\rm codim} U_\lambda$. The bottom row is the cohomology
long-exact sequence of the pair $(D_{J^+},D_J)$, where again $$H^{j-k_\lambda}(F_\lambda,\Z)\cong H^j(D_{J^+},D_J;\Z)$$ is the Thom isomorphism. Finally $i_\lambda:F_\lambda\to D_\lambda$, $i_J:D_J\to U_J$ and $i_{J^+}:D_{J^+}\to U_{J^+}$ denote the corresponding imbeddings. 

Clearly $i_\lambda^*$ is an isomorphism. So if we know that $i_J^*$ is an isomorphism, so will be $i^*_{J^+}$ by the five lemma. If $J_{min}=\{\lambda_{min}\}$ denotes a minimal element in $I$, then $D_{J_{min}}\cong F_{\lambda_{min}}$ and so $i_{J_{min}}^*:H^*(U_{J^{min}};\Z)\cong H^*(D_{J^{min}};\Z)$. Therefore by induction we get that $$i^*_J:H^*(U_{J};\Z)\cong H^*(D_{J};\Z)$$ is
an isomorphism for all $J$ satisfying \eqref{open}. Thus in particular we have:

\begin{theorem} \label{isocoh} The embedding $i:\calC\cong D_I\to X\cong U_I$ induces an isomorphism $i^*:H^*(X;\Z)\cong H^*(\calC;\Z)$.
\end{theorem}

\begin{corollary}\label{pure} A smooth semiprojective variety has pure cohomology. 
\end{corollary}

\begin{proof} As $X$ is non-singular all the non-trivial weights in $H^k(X;\Q)$ are at least $k$ by \eqref{smooth}.
By Theorem~\ref{isocoh}, Corollary~\ref{proper} and \eqref{projective} all the weights in $H^k(X;\Q)$ are at most
$k$. The statement follows\footnote{This argument is folklore yoga of weights; we learned it from G\'erard Laumon.}. 
\end{proof}

Interestingly our techniques can also be used to prove the purity of the cohomology of certain, typically affine, varieties which are deformations of semiprojective varieties as in the following corollary.

\begin{corollary}\label{ehresmann} Let $X$ be a non-singular complex
  algebraic variety and $f:X\to \C$ a smooth morphism, i.e. a
  surjective submersion. In addition, let $X$ be semiprojective with a
  $\C^\times$ action making $f$ equivariant covering a linear action
  of $\C^\times$ on $\C$ with positive weight.  Then the fibers
  $X_c:=f^{-1}(c )$ have isomorphic cohomology supporting pure mixed
  Hodge structures.\end{corollary}
 
 \begin{proof} The proof can be found in \cite[Appendix B]{hausel-aha1}. It proceeds by proving
 that the embedding of every fiber of $f$ induces an isomorphism \begin{equation}\label{song}H^*(X_c;\Q)\cong H^*(X;\Q),\end{equation} which implies the statement in light of Corollary~\ref{pure}. This is clear for $c=0\in \C$ as
 $X_c$ is itself semiprojective and it shares the same core $\calC\subset X_0\subset X$ with $X$. The proof of \eqref{song} for $0\neq c\in \C$ is more difficult and is using a version of the
 compactification technique as in \eqref{compact} and Ehresmann's theorem for proper smooth maps; in particular the proof is not algebraic. 
 \end{proof}

 \begin{remark} In fact Simpson's \cite{simpson-nonabelian} main
   example for a semiprojective variety was $\M_{\Hod}$, the moduli
   space of stable rank $n$ degree $1$ $\lambda$-connections on the
   curve which comes with $f:\M_{\Hod}\to \C$ satisfying the
   conditions of Corollary~\ref{ehresmann}. Here $f^{-1}(0)\cong
   \M_\Dol = \M_n^g$ is our moduli space of Higgs bundles while
   $f^{\lambda}=\M_\DR$ is the moduli space of certain holomorphic
   connections. The Corollary~\ref{ehresmann} then shows that
   $H^*(\M_{\DR};\Q)\cong H^*(\M_{\Dol};\Q)$ have isomorphic and pure
   cohomology. This argument was used in \cite[Theorem
   6.2]{hausel-thaddeus} and \cite[Theorem 2.2]{hausel-mln} in
   connection with topological mirror symmetry.
\end{remark}
\begin{remark} Another crucial use of this Corollary~\ref{ehresmann}
  is in our arithmetic harmonic analysis technique explained in~\S
  \ref{aha}. We will be able to compute the virtual weight polynomial
  $E(X_\lambda;q)$ of an affine symplectic quotient, and to deduce
  that it gives the Poincar\'e polynomial we will put $X_\lambda$ in a
  family $f:X\to \C$ satisfying the conditions of
  Corollary~\ref{ehresmann}. 
\end{remark}

The following result was discussed in \cite[Theorem 3.5]{hausel-sturmfels} in the context
of semiprojective toric varieties, and the proof was sketched in \cite{hausel-mathoverflow}.

\begin{corollary}\label{retract} The core $\calC$ is a deformation
  retract of the smooth semiprojective variety $X$.
\end{corollary}
\begin{proof} First we note that replacing cohomology with homology in
  the proof of Theorem~\ref{isocoh} yields that that
  $i_*:H_*(X;\Z)\cong H_*(\calC;\Z)$ induced by the embedding
  $i:\calC\to X$ is also an isomorphism. By the homology long exact
  sequence this is equivalent
  with \begin{equation}\label{relhom0}H_*(X,\calC;\Z)=0.\end{equation}

We also claim that $i_*:\pi_1(X)\cong \pi_1(\calC)$ induces an isomorphism on the fundamental group (from whose notation we omitted the base-point for simplicity). This follows by induction similarly as in the proof of Theorem~\ref{isocoh}. First note by \cite[Theorem 4.1]{bialynicki}
that $U_{\lambda_{min}}$ retracts to $F_{\lambda_{min}}\cong D_{\lambda_{min}}$ thus have isomorphic fundamental group. Then by induction we assume $(i_{J})_*:\pi_1(D_J)\cong \pi_1(U_J)$ is an isomorphism for an index set $J\subset I$ satisfying \eqref{open}.   Take $\lambda \in I\setminus J$ minimal and
cover $U_{J^+}=U_J\cup U_\lambda$ with open sets $U_J$ and a small tubular neighborhood $U^{tub}_\lambda$ of $U_\lambda$, small in the sense that $U^{tub}_
\lambda\cup D_J=\emptyset$ it is disjoint from the proper $D_J$ ($D_J$ is the core of the semiprojective $U_J$; thus proper by Theorem~\ref{proper}). This implies that $F_\lambda ^{tub}:=U^{tub}_\lambda\cap D_{J^+}\subset D_\lambda$ is a tubular neighborhood of $F_\lambda$. Then we have two commutative diagrams: \begin{equation}\label{pushout} \begin{array}{ccc} \pi_1(U^{tub}_\lambda\cap U_J)&\to& \pi_1(U^{tub}_\lambda) \\ \uparrow \cong&& \uparrow\cong
\\ \pi_1( F_\lambda^{tub}\cap (U_J\cap D_{J^+}))&\to& \pi_1(D_\lambda ^{tub}) \end{array} \hskip.5cm \begin{array}{ccc} \pi_1(U^{tub}_\lambda\cap U_J)&\to& \pi_1(U_J) \\ \uparrow\cong && \uparrow\cong
\\ \pi_1(F_\lambda^{tub}\cap (U_J\cap D_{J^+}))&\to& \pi_1(D_J) \end{array}  \end{equation}
where the maps are all induced by the embedding of the indicated varieties in each other. 
The four vertical arrows are all isomorphisms. The last one because of the induction hypothesis. 
The second one as both $U_\lambda^{tub}$ and $D_\lambda^{tub}$ retract to $F_\lambda$. Finally, the first and the third because these spaces all retract to $F_\lambda^{tub}\setminus F_\lambda$.

Using the diagrams \eqref{pushout} and the Seifert-van Kampen theorem applied to both the open covering $$U_{J^+}=U_\lambda^{tub} \cup U_J$$ and $$D_{J^+}=F^{tub}_\lambda \cup (U_J\cap D_{J^+})$$ we see that $$\pi_1(U_{J^+})\cong \pi_1({J^+}).$$ By induction we get the desired
$$\pi_1(X)\cong \pi_1(\calC).$$

In particular, the homotopy long exact sequence of the pair $(X,\calC)$ implies that $\pi_1(X,\calC)=0$ as well as that $\pi_2(X,\calC)$ is a quotient of $\pi_2(X)$ and so abelian. From this and \eqref{relhom0} the relative Hurewitz theorem \cite[Theorem IV.7.3]{whitehead} implies $\pi_k(X,\calC)=0$ for every $k$, thus
$$i_*:\pi_k(X)\stackrel{\cong}{\to} \pi_k(\calC)$$ is an isomorphism. Therefore $X$ and $\calC$ are weakly homotopy equivalent, and as varieties they are CW complexes and so by Whitehead's theorem \cite[Theorem V.3.5]{whitehead} $i$ is a homotopy equivalence.  
\end{proof}

\begin{theorem} The Bialinycki-Birula decomposition $X=\coprod_{i\in
    I} U_i$ of a semiprojective variety is perfect. In particular
  $P(X;t)=\sum_{\lambda\in I} P(F_\lambda;t) t^{2k_\lambda}$.
\end{theorem}

\begin{proof} This follows from studying the top long-exact sequence
  of \eqref{compare} considered with rational coefficients. Here we
  assume the same situation as there:
\begin{equation}\label{sequence} H^{q}(U_{J^+},U_J;\Q) \to  H^{q}(U_{J^+};\Q)\to H^{q}(U_J;\Q)\to H^{q+1}(U_{J^+},U_J;\Q). \end{equation} This is a sequence of Mixed Hodge structures, and the
weights are pure according to Corollary~\ref{pure} in the cohomology
of the semiprojective varieties $U_J$ and $U_{J^+}$, and in
$H^q(U_{J^+},U_{J};\Q)$ by the Thom isomorphism. Therefore the
connecting morphism $H^{q}(U_J;\Q)\cong W^q(H^q(U_j;\Q)) \to
H^{q+1}(U_{J^+},U_J;\Q)$ must be trivial. Therefore the long exact
sequence splits, the stratification is perfect, and the formula for
Poincar\'e polynomials follow by induction.
\end{proof}

\subsection{Weak Hard Lefschetz}
\label{whl}

Fix a very ample line bundle $L$ on a smooth semiprojective variety
$X$ and let $\alpha=c_1(L)\in H^2(X;\Q)$. Then we have

\begin{theorem}[Weak Hard Lefschetz] \label{weak} Let $X$ be a
  semiprojective variety $X$ with core $\calC=\cup_{\lambda\in I}
  D_\lambda$. Assume $\calC$ is equidimensional of pure dimension
  $k=\dim \calC.$ Then the Hard Lefschetz map 
\begin{equation}
\label{lefschetz}
\begin{array}{c}
    L^{i} :
    H^{k-i}(X,\Q)\to H^{k+i}(X,\Q)\\
    L^{i}(\beta)=\beta\wedge\alpha^{i}\end{array}
\end{equation}
is injective for $0\leq i<k$.
\end{theorem}

\begin{proof} It follows from Corollaries \ref{retract} and \ref{pure}
  that the core $\calC$ has pure cohomology. Then the result follows from
  \cite[Theorem 2.2]{bjorner-ekedahl} as we have assumed $\calC$ is
  equidimensional. Their argument goes by first showing that the
  natural map $H^*(\calC;\Q)\to IH^*(\calC;\Q)$ is injective, and then
  concludes by using \cite[Theorem 5.4.10]{bbd} for the Hard Lefschetz
  theorem for $IH^*(\calC;\Q)$.
\end{proof}

\begin{remark} An immediate consequence of the injectivity of
  \eqref{lefschetz} for $0\leq i<k$ are the inequalities 
\begin{equation}
\label{ineq}
  b_i(X)\leq b_{i+2j}(X) \mbox{ for all } 0\leq j \leq k-i 
\end{equation}
for the Betti numbers of the smooth semiprojective variety.  As a
consequence both sequences of odd and even Betti numbers grow until
$k$ and satisfy $b_{k-i}(X)\leq b_{k+i}(X)$.
\end{remark}

\begin{remark} Possibly the analogous result to \eqref{lefschetz} holds when $\calC$ is not
  equidimensional and $k$ is the smallest dimension of the irreducible components
  of $\calC$. It was proved in the case of smooth semiprojective toric
  varieties in \cite{hausel-sturmfels}. There however it was used that
  the components of the core are smooth; but conceivably this can be
  avoided.
\end{remark}

\begin{remark} Of course a general semiprojective toric variety could
  have a non-equidimensional core (as it corresponds to the complex of
  bounded faces of a non-compact convex polyhedron).  However, we do
  not know of an example of a semiprojective hyperk\"ahler variety
  whose core is not equidimensional.

When the semiprojective variety is hyper-compact
(Definition~\ref{hyper-compact})
%% symplectic and such that the symplectic form $\omega\in \Omega^2(X)$
%% is of homogeneity $1$ with respect to the $\C^\times$-action, in the
%% sense that $\lambda^*(\omega)=\lambda \omega$, 
one finds that $D_\lambda$ is Lagrangian. In other words, $\dim
D_\lambda=\frac{\dim X}{2}$ and hence $k=\frac{\dim X}{2}$ as ${\rm
  codim} U_{min}=0$. Examples include toric hyperk\"ahler manifolds,
Nakajima quiver varieties (from quivers without edge-loops) and the
moduli space of Higgs bundles. The fact that the nilpotent cone, which
agrees with the core of $\M^g_n$, is Lagrangian was first observed by
Laumon \cite{laumon}. Retrospectively, this can also be considered as
a consequence of the completely integrability of the Hitchin system
\cite{hitchin-integrable}. In the hyper-compact case
Theorem~\ref{weak} appeared as \cite[Corollary 4.3]{hausel-qgm}.

However, when the quiver contains an edge loop the Nakajima quiver
varieties are not hyper-compact. Examples include $(\C^2)^{[n]}$ and
more generally the ADHM spaces $\M_{n,m}$.  Nevertheless, in these
cases we know by \cite{briancon} and respectively
\cite{ellingsrud-lehn} and \cite{baranovsky} that the cores are
irreducible and in particular equidimensional of dimension $n-1$ and
 $mn-1$ respectively.

 We do not know if equidimensionality or even irreducibility of the
 core of Nakajima quiver varieties for quivers with edge loops 
 holds in general. 
 \end{remark}
 
 \begin{remark} In the case of smooth projective toric varieties $Y$,
   the Hard Lefschetz theorem, together with the fact that $H^2(Y)$
   generates $H^*(Y)$, famously \cite{stanley} gives a complete
   characterization of possible Poincar\'e polynomials of smooth
   projective toric varieties, and in turn the face vectors of
   rational simple complex polytopes.
 
   The above Weak Hard Lefschetz theorem was used in
   \cite{hausel-sturmfels} and \cite{hausel-qgm} to give new
   restrictions on the Poincar\'e polynomials of toric hyperk\"ahler
   varieties and, in turn, on the face vectors of rational hyperplane
   arrangements. However a complete classification in this case has not
   even been conjectured.
 \end{remark}

 \begin{remark} For the moduli space of Higgs bundles $\M^g_n$
   Theorem~\ref{weak} is a consequence of the Relative Hard Lefschetz
   theorem \cite{hausel-thmhcv} using the argument of
   \cite[4.2.8]{hausel-villegas}.
\end{remark}

Thus it is interesting to ask the following:

\begin{question} \label{quest} For semiprojective  hyperk\"ahler varieties is there a
  stronger form of the Weak Hard Lefschetz theorem or the inequalities
  \eqref{ineq}? In particular how do the Betti numbers of
  semiprojective hyperk\"ahler varieties behave after $k=\dim\calC$?
\end{question}

This question was the original motivation to look at the Betti numbers
of examples of large semiprojective hyperk\"ahler varieties to find
how the Betti numbers behave after the critical dimension $k=\dim
\calC$.

It turns out that partly due to an arithmetic harmonic analysis
technique to evaluate such Betti numbers we have now efficient
formulas to compute Poincar\'e polynomials. This allows us to
investigate numerically the shape of Betti numbers of large
seimprojective hyperk\"ahler manifolds in several examples. We explain
this arithmetic technique and the resulting combinatorial formulas for
the Poincar\'e polynomials in the next section.

\section{Arithmetic harmonic analysis on symplectic quotients:
  \newline the microscopic picture} 
\label{aha}
In the previous section we collected results on the cohomology of a
general semiprojective variety $X$. In this section we show that when
$X$ arises as symplectic quotient of a vector space, we can use
``arithmetic harmonic analysis" to count points on $X$ over a finite
field, and in turn to compute Betti numbers. Counting points of
varieties over finite fields is what we call microscopic approach to
study Betti numbers of complex algebraic varieties.

\subsection{Katz's theorem}

From Katz's \cite[Appendix]{hausel-villegas} we recall the definition
that a complex algebraic variety $X$ is strongly-polynomial
count. This means that there is polynomial ${\cal P}_X(t)\in \Z[t]$
and a spreading out\footnote{I.e. a homomorphism $R\to \C$ such that
  $X={\cal X}\otimes_R \C$.} $\cal X$ over a finitely generated
commutative unital ring $R$ such that for all homomorphism $\phi:R\to
\F_q$ to a finite field $\F_q$ (where $q=p^r$ is a power of the prime
$p$) we have $$\#{\cal X}_\phi(\F_q)={\cal P}_X(q).$$ We then have the
following theorem of Katz from \cite[Theorem 6.1,
Appendix]{hausel-villegas}:

\begin{theorem}[Katz] 
\label{katzt}
Assume that $X/\C$ is strongly-polynomial count with
counting polynomial ${\cal P}_{X}\in\Z[t]$. Then
$$
\label{katz}
E(X;q)={\cal P}_{X}(q).
$$
\end{theorem}

This result gives the Betti numbers of a strongly polynomial count variety $X$, when additionally
it has a pure cohomology. In that case \eqref{purel} will compute the Poincar\'e polynomial from
the virtual weight polynomial. This will be the case for many of our semiprojective varieties, where we will be able to use an effective technique to find the count polynomial $\calP_X(t)$. This technique 
from \cite{hausel-betti,hausel-kac} we explain in the next section. 
\subsection{Arithmetic harmonic analysis}
We work in the setup of \S \ref{hksp} but change coefficients from
$\C$ to a finite field $\F_q$.  For simplicity we denote with the same
letters $\rho,\G,\g,\M,\V,\mu$ the corresponding objects over
$\F_q$. We define the function $a_\varrho:\g\to {\mathbb N}\subset \C$
at $X\in \g$ as \begin{equation}\label{kernel}
  a_\varrho(X):=|\ker(\varrho(X))|. \end{equation} In particular
$a_\rho(X)$ is always a $q$ power.  Our main observation from
\cite{hausel-betti,hausel-kac} is the following:

\begin{proposition} Let $\xi\in \g^*$ and fix a non-trivial additive character $\Psi:\F_q\to \C^\times$. The number of solutions of the equation  $\mu(v,w)=\xi$
 over the finite field $\F_q$ equals:  \begin{eqnarray}\label{fourier} \# \{(v,w)\in \bM\,\, |\,\, \mu(v,w)= \xi\}&=&  |\g|^{-1} |\V| \sum_{X\in \g} a_{\varrho
}(X) \Psi(\langle X, \xi\rangle)\end{eqnarray}
\label{main}
\end{proposition} 

Thus in order to count the $\F_q$ points of $\mu^{-1}(\xi)$ we only
need to determine the function $a_\varrho$ as defined in
\eqref{kernel} and compute its Fourier transform as in
\eqref{fourier}.  In turn we assume that $\xi\in (\g^*)^\G$ and we use
this to count the $\F_q$ points of the affine GIT quotients
$X:=\mu^{-1}(\xi)/\!/ \G$, in cases when $\G$ acts freely on
$\mu^{-1}(\xi)$, when the number of $\F_q$ points on
$\mu^{-1}(\xi)/\!/ \G$ is just $\#\mu^{-1}(\xi)/|\G|$. In our cases
considered below this quantity will turn out to be a polynomial in
$q$, yielding by \eqref{katz} a formula for the virtual weight
polynomials of affine GIT quotient $\mu^{-1}(\xi)/\!/ \G$.
 
 Finally we can connect the affine GIT quotient to the GIT quotients
 with generic linarization as in \S \ref{hksp} by considering $X:=\mu^{-1}(\overline{\C^\times \xi})/\!/_\sigma \G$ a non-singular  semiprojective variety with a projection $f:X\to \C\cong \overline{\C^\times \xi} \subset \g^*$ with generic fiber $X_{\lambda \xi}=f^{-1}(\lambda\xi)=\mu^{-1}(\lambda\xi)/\!/_\sigma  \G=\mu^{-1}(\lambda\xi)/\!/  \G$ the affine GIT quotient when $\lambda\neq 0$ and 
 $X_0=\mu^{-1}(0)/\!/_\sigma  \G$ the GIT quotient with linearization $\sigma$. Now 
 Corollary~\ref{ehresmann} can be applied to show that $X_0$ and $X_\xi$ have isomorphic pure cohomology, and so our computation by Fourier transform above gives the Poincar\'e polynomial
 of our semiprojective varieties, which arise as finite dimensional linear symplectic quotients.

\subsection{Betti numbers of semiprojective hyperk\"ahler varieties}
\label{bettihk}
\subsubsection{Toric hypark\"ahler varieties $\M_\calH$}
Let $\calH\subset \Q^n$ be a rational hyperplane arrangement. In this
case the toric hyperk\"ahler variety $\M_\calH$ arises as linear
symplectic quotient, with generic linearization, induced by a torus
action $\rho_\calH:{\rm T}\to \GL(V)$ constructed from $\calH$ as in
\cite[\S 6]{hausel-sturmfels}. The variety $\M_\calH$ is an orbifold
and is non-singular when $\calH$ is unimodular. In the unimodular case it
was first constructed in \cite{bielawski-dancer} by differential
geometric means.

%{\bf Fixme: explain unimodal}

As explained in \cite{hausel-betti} the above arithmetic harmonic
analysis can be used to compute the Betti numbers of the
semiprojective $\M_\calH$; we get \begin{equation}\label{toricpoin}
  P(\M_\calH;t)=\sum h_i(\calH) t^{2i},\end{equation} where the Betti
numbers $h_i(\calH)$ are the $h$-numbers of the hyperplane arrangement
$\calH$; a combinatorial quantity.  In the unimodal case
\eqref{toricpoin} was first determined in \cite{bielawski-dancer} and
in the general case it was proved in \cite{hausel-sturmfels}.

As explained in \cite[\S 8]{hausel-sturmfels} one can construct the so-called cographic arrangement
 $\calH_Q$ from any graph $Q$. Then $\M_{\calH_Q}$ is just the Nakajima quiver variety
 $\M_\1^Q$ of \S \ref{nakajima} below. In this case the $h$-polynomial of \eqref{toricpoin} can be computed
 from the Tutte polynomial as follows: \begin{equation}\label{reliability}  P(\M_\1^Q;t)=t^{\dim \M_\1^Q } R_Q(1/t^2)=t^{\dim \M_\1^Q } T_Q(1,1/t^2),\end{equation} 
Here the {\it Tutte polynomial} $T_Q$ of a graph $Q$ is a two
variable polynomial invariant, universal with respect to
contraction-deletion of edges. It can be defined explicitly as follows
\begin{equation}
    T_Q(x,y):=\sum_{A\subseteq E}
    (x-1)^{k(A)-k(E)}(y-1)^{k(A)+\# A-\# V},
\end{equation}
where $k(A)$ denotes the number of connected components of the
subgraph $Q_A\subseteq Q$ with edge set $A$ and the same set
$V=V(Q)$ of vertices as $Q$. Note that the exponent $k(A)+\#
A-\# V$ equals $b_1(Q_A)$, the first Betti number of $Q_A$.

We will only consider the {\it external activity polynomial} $R_Q$ of $Q$
obtained by specializing to $x=1$. For $Q$ connected, we have
\begin{equation}
\label{R-pol-defn}
R_Q(q):=T_Q(1,q)=
\sum_{Q'\subseteq Q}(q-1)^{b_1(Q')},
\end{equation}
where the sum is over all connected subgraphs $Q'\subseteq
Q$ with vertex set $V$.  (This polynomial is related to the
reliability polynomial of $Q$ by a simple change of variables,
hence the choice of name.)  A remarkable theorem of Tutte guarantees
that $T_Q$, and hence also $R_Q$, has non-negative (integer)
coefficients.

For example, the Tutte polynomial of complete graphs $K_n$ was
computed in \cite{tutte} cf. also \cite[Theorem 4.3]{ardila}. This
implies the following generating function of the Poincar\'e
polynomials $P(\M_\1^{K_n};t)=R_{K_n}(t^2)$ of Nakajima toric quiver
varieties attached to the complete graphs $K_n$
\begin{equation}
  \label{complete} 
\sum_{n\geq 1} R_{K_n}(q)\frac{T^n}{n!}=(q-1) \log \sum_{n\geq 0}
q^{\binom n 2}\ \frac{(T/(q-1))^n}{n!}
\end{equation}

\subsubsection{Twisted ADHM spaces $\M_{n,m}$ and Hilbert  scheme of points on the affine plane $(\C^2)^{[n]}$}
\label{adhms}
Here $\G=\GL(V)$, where $V$ is an $n$-dimensional $\K$ vector space\footnote{Here $\K=\C$ when
we study the complex semiprojective varieties and $\K=\F_q$ when we do arithmetic harmonic
analysis on them.}.
We need three types of basic representations of $\G$. The adjoint
representation $\rho_{ad}: \GL(V)\to \GL(\gl(V))$, the defining
representation $\rho_{def}=Id:\G\to \GL(V)$ and the trivial
representations $\rho^W_{triv}=1:\G\to \GL(W)$, where
$\dim_\K(W)=m$. Fix $m$ and $n$. Define $\V=\gl(V)\times \Hom(V, W)$,
$\M=\V\times \V^*$ and $\rho:\G\to \GL(\V)$ by $\rho=\rho_{ad}\times
\rho_{def}\otimes \rho^W_{triv}$. Then we take the central element
$\xi=Id_V\in \gl(V)$ and define the twisted ADHM space as
$$\M_{n,m}=\M/\!/\!/\!/\!_\xi \G=\mu^{-1}(\xi)/\!/\G,$$  where 
$$\mu(A,B,I,J)=[A,B]+IJ,$$ with $A,B\in \gl(V)$, $I\in \Hom(W, V)$ and
$J\in \Hom(V,W)$. 

The space $\M_{n,m}$ is empty when $m=0$ (the trace of a commutator is
always zero), diffeomorphic with the Hilbert scheme of $n$-points on
$\C^2$, when $m=1$, and is the twisted version of the ADHM space
\cite{atiyah-etal} of $U(m)$ Yang-Mills instantons of charge $n$ on
$\R^4$ (c.f. \cite{nakajima-hilbert}). As explained in \cite[Theorem
2]{hausel-betti} the arithmetic Fourier transform technique of \S
\ref{aha} yields the following generating function for the Poincar\'e
polynomials of $\M_{n,m}$ (originally due to \cite[Corollary
3.10]{nakajima-yoshioka}):

\begin{equation}
\label{adhm}
\sum_{n=0}^\infty P(\M_{n,m};t) T^n=\prod_{i=1}^\infty \prod_{b=1}^m
\frac{1}{(1-t^{2\left( m(i-1)+b-1\right)}T^i)}.
\end{equation}

In particular when $m=1$ this gives for the generating function of
Poincar\'e polynomials of Hilbert schemes of points on $(\C^2)^{[n]}$
\begin{equation}
\label{hilbert}
\sum_{n=0} P(({\C^2})^{[n]};t) T^n=\prod_{i=1}^\infty
\frac{1}{(1-t^{2\left( i-1\right)}T^i)}, 
\end{equation}
G\"ottsche's formula from \cite{gottsche}, which by Euler's formula
reduces to 
\begin{equation}
\label{number parts}
b_{2i}\left(({\C^2})^{[n]}\right)=\#\left\{\lambda  \; | \;  |\lambda|
  =n, l(\lambda)=i  \right\}
\end{equation}
where $l(\lambda)$ is the number of parts in the partition $\lambda$
of $n$; this was the original computation of Ellingsrud-Stromme in
\cite{ellingsrud-stromme}.

\subsubsection{Nakajima quiver varieties $\M_{\v,\w}$ and $\M_{\v}$ }
\label{nakajima}
Here we recall the definition of the affine version of Nakajima's 
quiver varieties \cite{nakajima}. Let $Q=(\calV,\calE)$ be a quiver, i.e. an oriented 
graph on a finite set $\calV=\{1,\dots,n\}$ with $\calE\subset \calV\times \calV$ a finite set of oriented
(perhaps multiple and loop) edges. To each vertex $i$ of the graph we associate two finite dimensional $\K$ vector
spaces $V_i$ and $W_i$. We call $(\v_1,\dots,\v_n,\w_1,\dots, \w_n)=(\v,\w)$ the dimension vector, where $\v_i=\dim(V_i)$ and $\w_i=\dim(W_i)$. To this data we associate the grand vector space:
$$\V_{\v,\w}=\bigoplus_{(i,j)\in \calE} \Hom(V_i,V_j) \oplus \bigoplus_{i\in \calV} \Hom (V_i,W_i),$$
the group and its Lie algebra $$\G_{\v}=\varprod_{i\in \calV} \GL(V_i)$$ $$ \g_{\v}=\bigoplus_{i\in \calV} \gl(V_i),$$ and the natural representation 
$$\rho_{\v,\w}:\G_{\v}\to \GL(\V_{\v,\w}),$$ with derivative $$\varrho_{\v,\w}:\g_{\v}\to \gl(\V_{\v,\w}).$$ The action is  from both left and right
on the first term, and from the left on the second.  

We now have $\G_{\v}$ acting on $\bM_{\v,\w}=\V_{\v,\w}\times
\V_{\v,\w}^*$ preserving the symplectic form with moment map
$\mu_{\v,\w}:\V_{\v,\w}\times\V_{\v,\w}^*\to \g_{\v}^*$ given by
(\ref{momentmap}).  We take now $\xi_\v=(Id_{V_1},\dots,Id_{V_n})\in
(\g_{\v}^*)^{\G_{\v}}$, and define the affine Nakajima quiver variety
\cite{nakajima} as
$$\M_{\v,\w}=\mu_{\v,\w}^{-1}(\xi_\v)/\!/\G_{\v}.$$

As explained in \cite{hausel-betti} and \cite{hausel-kac} the
arithmetic harmonic analysis technique of \S \ref{aha} translates to
the formula~\eqref{quivergenerate} below. We first introduce some
notation on partitions following~\cite{macdonald}.  We let $\calP(s)$
be the set of partitions of $s\in \N$.  For two partitions
$\lambda=(\lambda_1,\dots,\lambda_l)\in \calP(s)$ and
$\mu=(\mu_1,\dots,\mu_m)\in \calP(s)$ we define
$n(\lambda,\mu)=\sum_{i,j} \min(\lambda_i,\mu_j)$.  Writing
$\lambda=(1^{m_1(\lambda)},2^{m_2(\lambda)},\dots)\in \calP(s)$ we let
$l(\lambda)=\sum m_i(\lambda)=l$ be the number of parts in $\lambda$.
For any $\lambda\in\calP(s)$ we have $n(\lambda,(1^s))=s l(\lambda)$.

\begin{theorem} Let $Q=(\calV,\calE)$ be a quiver, with $\calV=\{1,\dots,n\}$ and $\calE\subset \calV\times \calV$, with possibly multiple edges and loops. Fix a dimension vector  $\w\in \N^\calV$. The Poincar\'e polynomials $P(\M_{\v,\w})$ of the corresponding Nakajima quiver varieties are given by the generating function:
\begin{equation}
\label{quivergenerate} 
 \sum_{\v\in \N^\calV} P_t(\M_{\v,\w}) t^{-d(\v,\w)}T^\v=
\frac{\sum_{\v\in \N^\calV} T^\v \sum_{\lambda^1\in 
\calP(\v_1)}\dots \sum_{\lambda^n\in \calP(\v_n) } \frac{
\left( \prod_{(i,j)\in \calE} t^{-2n(\lambda^i,\lambda^j)}\right)\left(\prod_{i\in \calV} t^{-2n(\lambda^i,(1^{\w_i}))}\right) }{\prod_{i\in \calV} \left(t^{-2n(\lambda^i,\lambda^i))}\prod_k \prod_{j=1}^{m_k(\lambda^i)} (1-t^{2j}) \right)}}{\sum_{\v\in \N^\calV} T^\v \sum_{\lambda^1\in 
\calP(\v_1)}\dots\sum_{{\lambda^n}\in \calP(\v_n)}
\frac{\prod_{(i,j)\in \calE}
  t^{-2n(\lambda^i,\lambda^j))}}{\prod_{i\in \calV}
  \left(t^{-2n(\lambda^i,\lambda^i))}\prod_k
    \prod_{j=1}^{m_k(\lambda^i)} (1-t^{2j})\right) }},\end{equation}
where $d(\v,\w)=2{{\sum_{(i,j)\in \calE} \v_i\v_j}+2{\sum_{i\in \calV}
    \v_i(\w_i-\v_i)}}$ is the dimension of $\M_{\v,\w}$ and
$T^\v=\prod_{i\in \calV} T_i^{\v_i}$. 
\end{theorem}

\begin{example} We will look at the case of $Q=A_1$ a single vertex with no edges. In this case the semiprojective quiver
variety $\M^{A_1}_{n,m}$ is isomorphic \cite[Theorem 7.3]{nakajima-instantons} with the cotangent bundle $T^*\Gr(n,m)$  of the Grassmanian of $m$ planes in $\C^n$. In this case we can
directly count points on $\Gr(n,m)$ over finite fields and we get the following well known formula for its Poincar\'e polynomial:
\begin{equation}\label{grassmannpoin} P(T^*\Gr(n,m);t)=\left[ \begin{array}{c} n \\ {k} 
\end{array}\right]_{t^2} = {\displaystyle \prod_{i=1}^k
\frac{1-t^{2(n+1-i)}}{1-t^{2i}}}\end{equation}\end{example}
\noindent
The combination of \eqref{grassmannpoin} and \eqref{quivergenerate}
gives a curious $q$-binomial type of theorem.

\begin{example} When the quiver is the Jordan quiver, i.e. one loop on a single vertex, then
$\M_{\v,\w}=\M_{n,m}$ the twisted Yang-Mills moduli spaces from \S \ref{adhms}. The formula
\eqref{quivergenerate} then reduces to \eqref{adhm}.
\end{example}

We also consider Nakajima quiver varieties $\M_\v$ attached to a single dimension vector
$\v=(\v_1,\dots,\v_n)$ on the same quiver $Q$. We construct 
$$\V_{\v}:=\bigoplus_{(i,j)\in \calE} \Hom(V_i,V_j),$$ which will also carry a natural
representation $$\rho_{\v,\w}:\G_{\v}\to \GL(\V_{\v,\w}).$$ In the
framework of \eqref{hksp} this gives rise to the symplectic vector
space $\bM_\v:=\V_\v\times \V_\v^*$ and the moment map
$\mu_\v:\bM_\v\to\g_\v^*$ leading to the quotient
$\M_\v:=\mu^{-1}(0)//_\xi \G_\v$ where $\xi\in \Hom(\G_\v,\C^\times)$
is a character of $\G_\v$. When $\v$ is indivisible (i.e. the equation
$\v=k\v^\prime$ for an integer $k$ and dimension vector $\v^\prime$
implies $k=1$) it is known that $\xi$ can be chosen so that $\M_\v$ is
smooth semiprojective hyperk\"ahler variety. For $\v$  indivisible
it is proved in \cite{crawley-boevey-etal} 
that
 \begin{equation}\label{cb}P(\M_\v;t)=t^{d_\v}
  A_Q(\v;t^{-2}),
\end{equation}
where $A_Q(\v;q)\in \Z[q]$ is the {\it Kac polynomial} of
$Q,\v$~\cite{kac}, which counts absolutely indecomposable
representations of the quiver $Q$ and $d_\v=\dim \M_\v$.  A generating
function formula was obtained for $A_Q(\v;q)$ by Hua \cite{hua}, and
it takes the following combinatorial form:

\begin{equation}\label{hua} \sum_{\v\in \N^r\backslash\{0\}} A_\v(q)
T^\v=(q-1)\Log\left( \sum_{\pihat=(\pi^1,\dots,\pi^r) \in \calP^r}
  \frac{\prod_{i\rightarrow j\in \Omega} q^{\langle
      \pi^i,\pi^j\rangle}} {\prod_{i\in I}
    q^{\langle\pi^i,\pi^i\rangle}
    \prod_{k}\prod_{j=1}^{m_k(\pi^i)}(1-q^{-j})
  }T^{|\pihat|}\right).\end{equation}
  
    The formula \eqref{cb} was proved by the arithmetic harmonic analysis technique of \S \ref{aha} in \cite{hausel-dtkac} for all quivers, with a cohomological interpretation of $\A_Q(\v,q)$ in the case when $\v$ is a divisible dimension vector, settling Kac's conjecture
    \cite{kac} that the $A$-polynomial $\A_Q(\v,q)\in \Z_{\geq 0}[q]$ has non-negative coefficients.

\subsubsection{Moduli of Higgs bundles $\M^g_n$} Denote by $\M^g_n$ the moduli spaces of rank $n$ degree $1$ stable Higgs bundles on a smooth projective curve of genus $g$. The construction of the moduli space can be done by algebraic geometric techniques using GIT quotients as in \cite{nitsure} or by gauge theoretical means using an infinite dimensional hyperk\"ahler 
quotient construction as was done in the original paper \cite{hitchin}. This latter construction is not algebraic, and
so it is unclear how our arithmetic harmonic analysis of \S \ref{hksp} would extend to this case. 
The cohomology  of $\M^g_n$ is the most interesting (due to various connections to a variety 
of subjects) and the least understood. There are various results on its Betti numbers available
\cite{hitchin, gothen, heinloth-etal}
but we only have a conjectured formula. First we introduce rational functions $H_n(z,w)\in \Q(z,w)$ by the generating function:
\begin{equation}\label{mhpconj} \sum_{n=0}^\infty H_n(z,w)T^n=(1-z)(1-w)\Log\left( \sum_{\lambda\in \calP} \prod  \frac{\left(z^{2l+1}-w^{2a+1}\right)^{2g}}{(z^{2l+2}-w^{2a})(z^{2l}-w^{2a+2})}\,T^{|\lambda|} \right) \end{equation} Then we have the following conjecture \cite[Conjecture 4.2.1]{hausel-villegas}
\begin{equation}\label{poinconj} P(\M^g_n;t)=t^{d_n} H_n(1,-1/t),\end{equation}
where already part of the conjecture is that $H_n(w,z)\in \Z[w,z]$ is a polynomial in $w,z$. 

\begin{remark} This conjecture was obtained via a more elaborate version of the arithmetic harmonic 
analysis technique of \S\ref{hksp}. Namely a non-abelian version of the arithmetic harmonic analysis
allows us \cite{hausel-villegas} to count points on certain $\GL_n$-character varieties; and the conjecture \eqref{mhpconj}
is a non-trivial extension of that result, and the non-abelian Hodge theorem of \cite{simpson} which
shows that this $\GL_n$-character variety is canonically diffeomorphic with $\M^g_n$ thus shares its cohomology. \end{remark}

\begin{remark} As $\M^g_n$ is semiprojective by \cite[Corollary
  10.3]{simpson-nonabelian} its cohomology is pure by
  Corollary~\ref{pure}. Therefore counting the $\F_q$ rational points
  of $\M^g_n$ would lead to its Betti numbers. However we do not know
  how to extend our arithmetic harmonic analysis of \S \ref{hksp} to
  this case. There are recent works of Chaudouard and Laumon
  \cite{chaudouard-laumon,chaudouard} where a different kind of
  harmonic analysis is used to count $\#\M^g_n(\F_q)$ but so far only
  the $n=3$ case is complete where those results confirm the
  conjecture \eqref{mhpconj}.
\end{remark}

\begin{remark} One last observation is that the similarity of
  \eqref{hua} and \eqref{mhpconj} is not accidental. In fact it was
  proved in \cite[Theorem 4.4.1]{hausel-villegas} that
  $H_n(0,\sqrt{q})=A_{S_g}((n),q)$ where $S_g$ is the $g$ loop quiver
  on one vertex. In particular a certain subring of $H^*(\M^g;\Q)$ is
  conjectured to have graded dimensions whith Poincar\'e polynomial
  $A_{S_g}((n),q)$. This and more general versions of such conjectures
  \cite[Conjecture 1.3.2]{hausel-aha2} show that the cohomology of
  Nakajima quiver varieties for comet-shaped quivers should be
  isomorphic with subrings of the cohomology of certain Higgs moduli
  spaces. This maybe relevant when we compare the large scale
  asymptotics of the Betti numbers of these varieties, as will be done
  in the remaining of this paper.
  \end{remark}

 \section{Visual distribution of Betti numbers; \newline the big picture}
 
 Motivated by Question~\ref{quest} in this section we will be studying pictures of the distribution of Betti numbers of 
 our semiprojective hyperk\"ahler manifolds. The reason we can look at very large examples are the combinatorially
 tractable formulas in the previous \S \ref{bettihk}.
 
 \subsection{Toric quiver variety $\M^{K_{40}}_{\bf 1}$}

 Using formula \eqref{complete} one can efficiently compute the Betti numbers of the toric quiver variety $\M^{K_{40}}_{\bf 1}$ 
 for the complete graph $K_{40}$ on $40$ vertices. This is a hyper-compact semiprojective 
 hyperk\"ahler variety of real dimension dimension   $2964$.  The top non-trivial Betti number therefore is the middle one $b_{1482}\approx 2 \times 10^{46}$. The sequence of Betti numbers is 
 {\em unimodal}\footnote{meaning it has single local maximum}  and the largest Betti number is
 $b_{1288}\approx 8\times 10^{58}$. In Figure~\ref{toricK40} we plotted only the non-trivial even Betti numbers\footnote{This means that  one needs to double the value on the $x$-axis to get the correct
degree for the Betti number. 
}. 
 \begin{figure}[H] 
\begin{center}
\includegraphics[height=7cm,width=12cm]{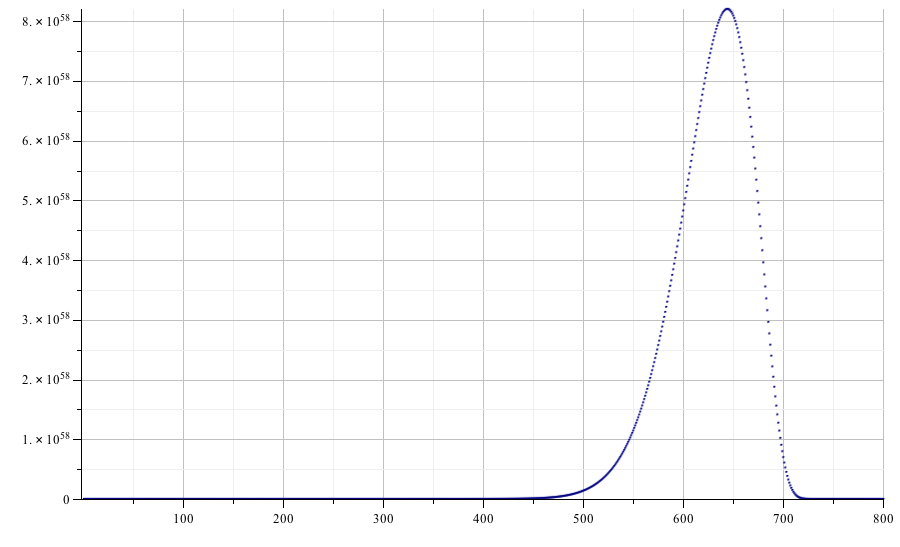}
\caption{Distribution of even Betti numbers of the toric quiver variety $\M^{K_{40}}_{\bf 1}$\label{toricK40}}
\end{center}
\end{figure}
\subsection{Hilbert scheme $(\C^2)^{[500]}$}

We can efficiently compute Betti numbers of Hilbert schemes of $n$ points on $\C^2$ for large $n$
using \eqref{hilbert}. When $n=500$ Figure~\ref{C2_500_} shows the distribution
of even Betti numbers. We have $\dim_\R (\C^2)^{[500]}=2000$. The Hilbert scheme $(\C^2)^{[500]}$ is a semiprojective hyperk\"ahler manifold, but not hyper-compact, and the top non-trivial Betti number is $b_{998}=1$. Again the sequence of Betti numbers is unimodal. The maximal Betti number is $b_{896}\approx 5.5\times 10^{19}$.

\begin{figure}[H] 
\begin{center}
\includegraphics[height=7cm,width=12cm]{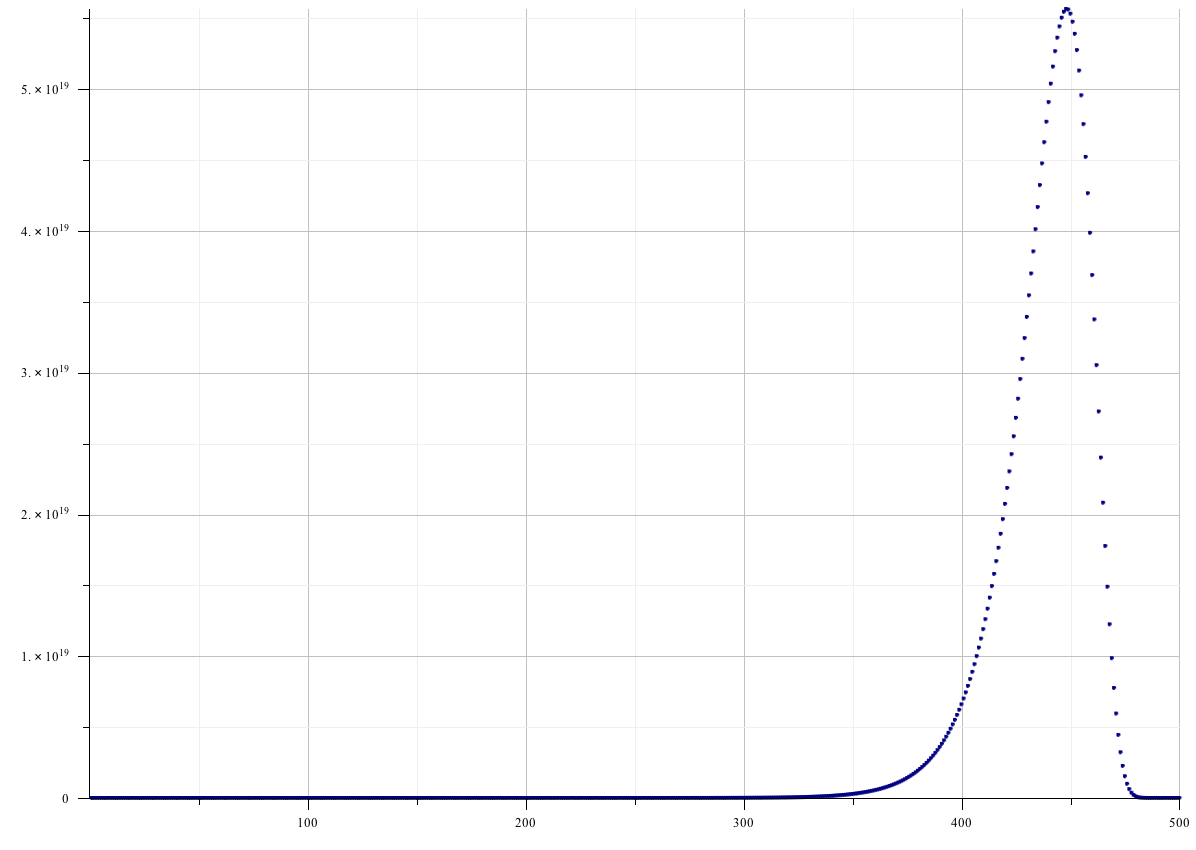}
\caption{Distribution of even Betti numbers of the Hilbert scheme $(\C^2)^{[500]}$ of $500$ points on $\C^2$\label{C2_500_}}
\end{center}
\end{figure}

 \subsection{Twisted ADHM space $\M^{\hat{A}_0}_{40,20}$}
 \begin{figure}[H] 
\begin{center}
\includegraphics[height=7cm,width=12cm]{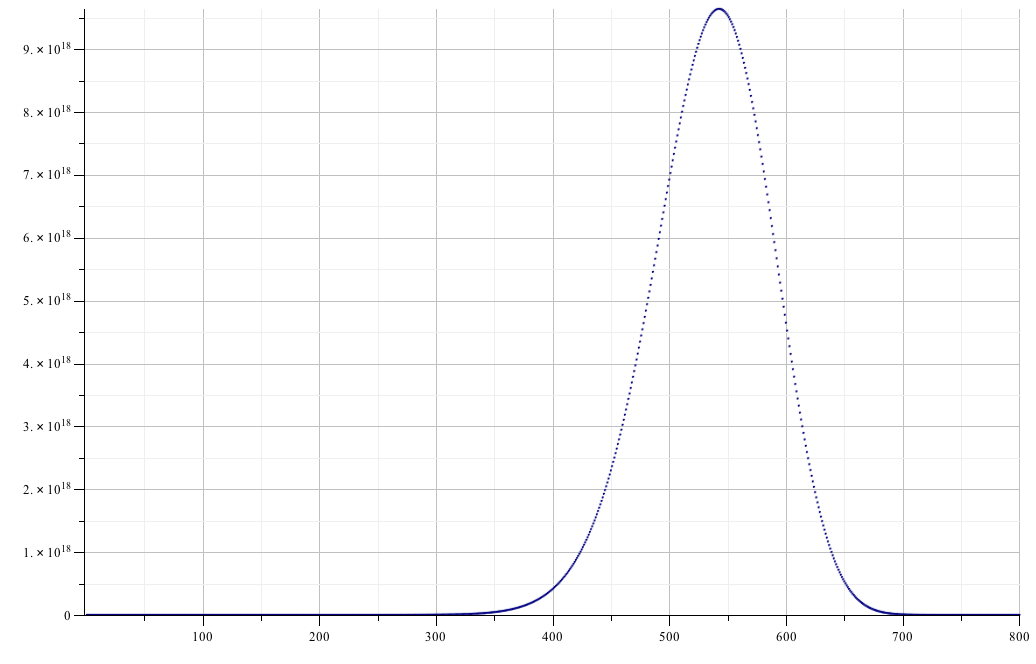}
\caption{Distribution of even Betti numbers of  ADHM space $\M^{\hat{A}_0}_{40,20}$\label{M(40,20)}}
\end{center}
\end{figure}

The Nakajima quiver variety $\M^{\hat{A}_0}_{m,n}$ attached to the Jordan quiver $\hat{A}_0$ and dimension vectors $\v=(m)$ and $\w=(n)$ is a semiprojective hyperk\"ahler variety, which is not hyper-compact.  When $m=40$ and $n=20$ Figure~\ref{M(40,20)} shows the distribution of even Betti numbers. The top non-zero Betti number is $b_{1598}=1$. There are only even Betti numbers and they form a unimodal sequence. The maximal Betti number is $b_{1086}\approx 9.6 \times 10^{17}$.

 \subsection{Cotangent bundle of Grassmannian $\M^{A_1}_{30,100}\cong T^*{\rm Gr}(100,30)$}
 \begin{figure}[htb] 
\begin{center}
\includegraphics[height=7cm,width=12cm]{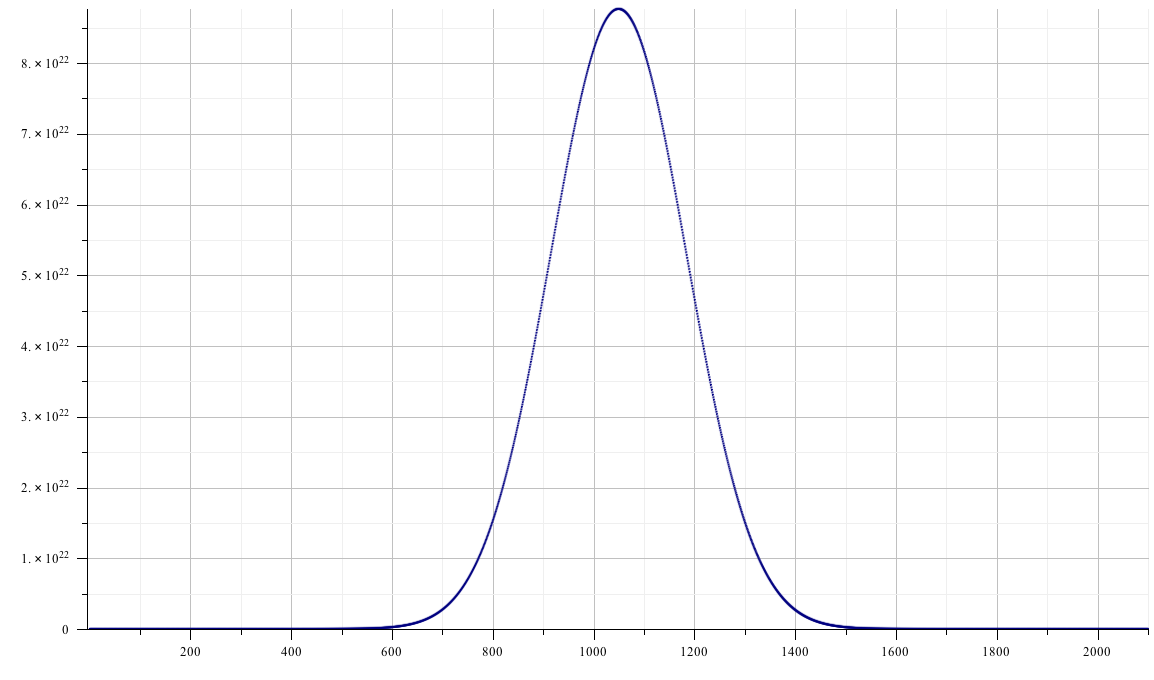}
\caption{Distribution of even Betti numbers of  cotangent bundle to Grassmannian $T^*{\rm Gr}(100,30)$\label{T*Gr(100,30)}}
\end{center}
\end{figure}

As was discussed earlier in \S \ref{nakajima} the Nakajima quiver variety $\M^{A_1}_{30,100}\cong T^*{\rm Gr}(100,30)$ for the trivial $A_1$ quiver with dimension vectors
$\v=(30)$ and $\w=(100)$ is the cotangent bundle to the Grassmannian of $30$ dimensional subspaces in $\C^{100}$. This is a semiprojective hyperk\"ahler manifold which is hyper-compact. Of course in this case the core is  the zero section of the cotangent bundle, thus it is the smooth projective Grassmannian ${\rm Gr}(100,30)$. 
It only has even cohomology and satisfies Hard Lefschetz. In particular the sequence of even Betti numbers is unimodal and symmetric. The top non-zero Betti number is $b_{4200}=1$ while the maximal one is $b_{2100}\approx 8.7 \times 10^{22}$. 

\subsection{A quiver variety $\M^{Q}_{(15,7)}$}
\begin{figure}[H] 
\begin{center}
\includegraphics[height=7cm,width=12cm]{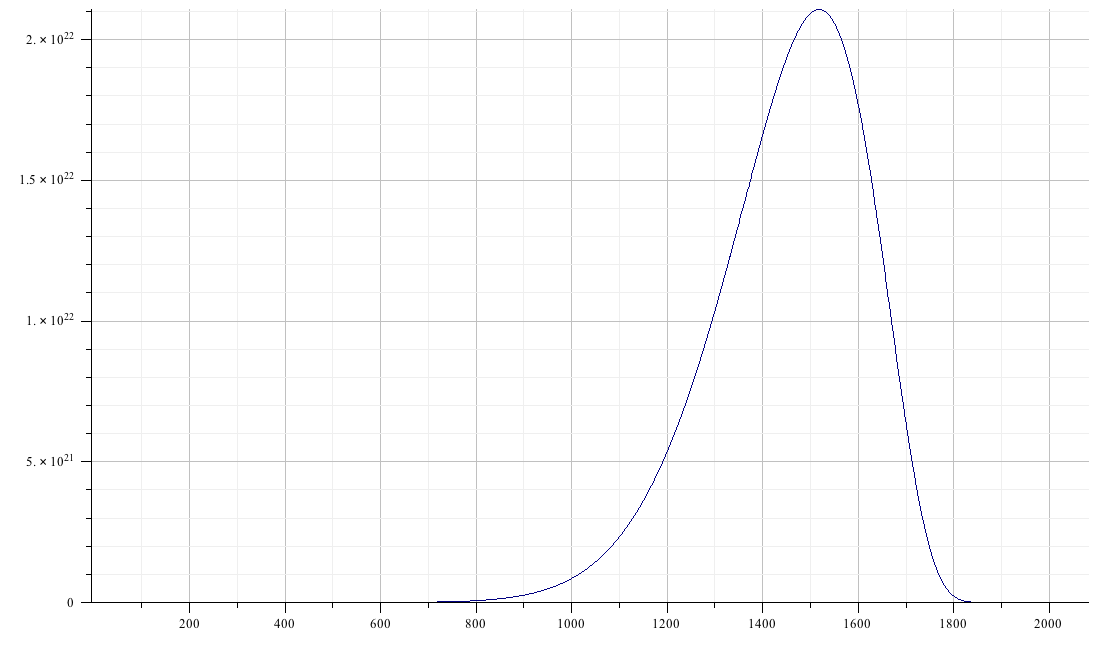}
\caption{Distribution of even Betti numbers of the quiver variety $\M^{Q}_{(15,7)}$\label{Huag10n5m7}}
\end{center}
\end{figure}

 We include a smooth quiver variety of type $\M^Q_\v$ where $Q$ is
the graph on two vertices, with $10$ loops on the first vertex, and a connecting edge to the second vertex, and furthermore $\v=(15,7)$. This is a smooth (because $(15,7)$ is indivisible) semiprojective hyperk\"ahler variety, which is not hypercompact, due to the presence of loops on the first vertex. We have $\dim_\R \M^{Q}_{(15,7)}=8328$ and top non-trivial Betti number 
 $b_{3862}=1$. Again, there are only even Betti numbers and their sequence  is unimodal, with
 maximal Betti number $b_{3036}\approx 2.1\times 10^{22}$.

 \subsection{Cotangent bundle of Jacobian $\M^{100}_{1}\cong T^*{\rm Jac}(C_{100})$}
 \begin{figure}[H] 
\begin{center}
\includegraphics[height=7cm,width=12cm]{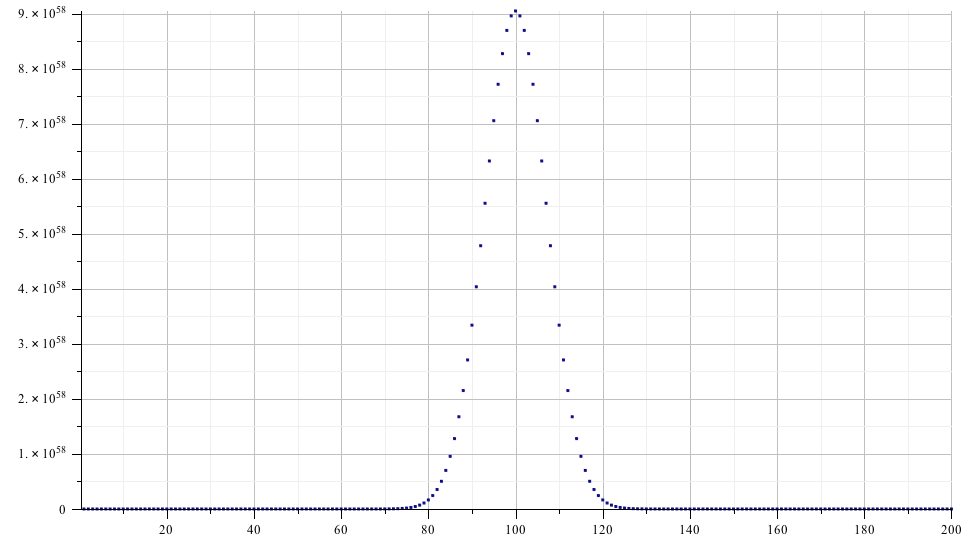}
\caption{Distribution of all Betti numbers of  cotangent bundle to Jacobian $T^*\rm Jac(C_{200})\label{TJacC100}$}
\end{center}
\end{figure}
When $n=1$ the moduli space of rank $1$ degree $1$ Higgs bundles is isomorphic with
$\M^g_1\cong T^* Jac(C_g)$ the cotangent bundle to the Jacobian of the curve $C_g$ of genus $g$. Of course this is also a semiprojective hyperk\"ahler manifold, which is hyper-compact. Just
like in the Grassmannian case above, the Jacobian $Jac(C_g)$ as the zero section of its cotangent bundle is the core of the semiprojective variety, that is the core is a smooth projective variety. The Poincar\'e polynomial then is just $$P(\M^g_1;t)=(1+t)^{2g}.$$ Figure~\ref{TJacC100} shows the distribution of all non-trivial Betti numbers when $g=100$. The top non-trivial one is $b_{200}=1$. 
The sequence of Betti numbers is unimodal, with maximal value $b_{100}\approx 8.7\times 10^{22}$.

 \subsection{Moduli space of Higgs bundles $\M^2_8$}
 \begin{figure}[H]
\begin{center}
\includegraphics[height=7cm,width=12cm]{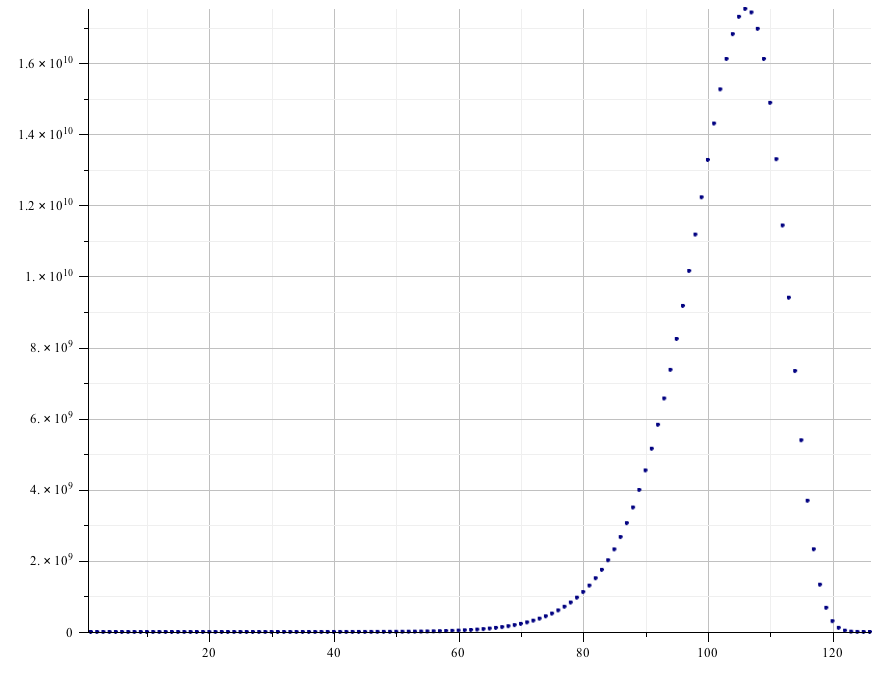}
\caption{Distribution of all Betti numbers of moduli of rank $8$ Higgs bundles $\M^2_8$ \label{hqtn8g2}}
\end{center}
\end{figure}

The moduli space of rank $n$ degree $1$ stable Higgs bundles on a smooth projective
curve of genus $g$ is smooth semiprojective, hyper-compact hyperk\"ahler manifold. We can use \eqref{mhpconj} and \eqref{poinconj} to compute the conjectured Betti numbers of $\M^g_n$ for small values of $n$ and $g$. In fact this formula is the most computationally demanding, and we could only evaluate the $g=2$ and $n=8$ case.   Part of the reason of the computational difficulty is  because the calculation goes through evaluating the two variable polynomial $H_8(q,t)$ from \eqref{mhpconj} which already has 11786 terms. At any rate in 
this particular case $\dim_\R \M^2_8=252$, thus the top non-trivial Betti number is $b_{126}$ which equals $12300$. An important difference between this case and the previous ones, is that $\M^2_8$ has non-triv
ial odd Betti numbers. The full sequence is not unimodal but both sequences of odd and even Betti numbers are unimodal. The maximal Betti number is $b_{106}\approx 1.7 \times 10^{10}$.

\section{Asymptotic shape of Betti numbers:\newline the macroscopic picture}

In the previous section we have plotted the distributions of Betti
numbers of some large examples of our semiprojective hyperk\"ahler
varieties. Originally we were motivated by studying potential
extensions of the Weak Hard Lefschetz
Theorem~\ref{weak}. Surprisingly, the plots in the previous section
behave in some peculiar manner. First we can note that the sequence of
even Betti numbers is always unimodal. Second and more puzzling is the
apparent asymptotic behavior of the distribution of Betti numbers: the
plots above seem to suggest the existence of a certain continuous
limiting distribution\footnote{Even a $C^\infty$ one: one can also
  plot the higher discrete derivatives of the distributions above and
  still get some continuous looking distributions}. Furthermore it
seems that the distributions on Figures \ref{TJacC100},
\ref{T*Gr(100,30)} and \ref{M(40,20)} are the same while the remaining
ones on Figures \ref{toricK40}, \ref{C2_500_}, \ref{Huag10n5m7} and
\ref{hqtn8g2} also look similar.

In this section we will prove some rigorous results about such
limiting distributions, in particular we will determine this
distribution in the case of Figures \ref{TJacC100},
\ref{T*Gr(100,30)}, \ref{C2_500_} and \ref{toricK40} and offer
conjectures in the remaining cases.

First we discuss what we mean by a limiting distribution of Betti
numbers of a family of varieties.   

\subsection{Generalities}
In this section we consider sequences of varieties $X_0,X_1,\ldots$
whose Betti numbers $b_i(X_n)$ approach a limiting distribution as
$n\rightarrow \infty$.  For simplicity, we will typically assume that
the all varieties $X$ under consideration satisfy $b_{2i+1}(X)=0$ and
let
$$
E(X,q):=\sum_{i=0}^{\dim X}b_{2i}(X)\ q^{\dim X-i}.
$$ 
If $X$ is a polynomial count variety with pure mixed Hodge structure
then by Theorem~\ref{katzt} $E(X,q)=\calP(q)$, where $\calP$ is a
polynomial such that $\#X(\F_q)=\calP(q)$ for generic $q$.

To a Laurent polynomial $E(q)=\sum_i e_i q^i$ with non-negative real
coefficients we associate the discrete measure $d\mu_E$ on
$[-\infty,\infty]$ such that
$$
\int_{-\infty}^\infty \phi(x) \ d\mu_E := \sum_i \phi(i)e_i.
$$

\begin{definition}
  Given a measure $\mu$ on $[-\infty,\infty]$ its {\it moments} are
  the real numbers
$$
M_k:=\int_{-\infty}^\infty   x^k\,d\mu
$$
and its {\it factorial moments} are the
real numbers
$$
m_k:=\int_{-\infty}^\infty \binom x k \,d\mu.
$$
\end{definition}
Clearly these two kinds of moments are linearly related and since the
leading term in $\binom x k$ is $x^k/k!$, typically the asymptotic
behavior of $m_k$ and $M_k/k!$ for a sequence of measures is the
same. It depends on the situation which set of moments is easier to
deal with.

For a measure $d\mu$ on $[-\infty,\infty]$ we define the generating
function of moments
\begin{equation}
\label{mom-gf}
M_\mu(t):=\sum_{k\geq 0} M_k \frac{t^k}{k!}, \qquad \qquad
m_\mu(\eta):=\sum_{k\geq 0} m_k\ \eta^k.
\end{equation}
If $d\mu_E$ is a  discrete measure associated to the Laurent polynomial
$E=\sum_i e_i q^i$ then
\begin{equation}
\label{mom-gf-E}
M_{\mu_E}(t)=E(e^t), \qquad m_{\mu_E}(\eta)=E(1+\eta).
\end{equation}
If $d\mu_1, d\mu_2$ are two measures then 
$$
M_{\mu_1}(t)M_{\mu_2}(t)=M_\mu(t), 
\qquad \qquad 
m_{\mu_1}(t)m_{\mu_2}(t)=m_\mu(t),
$$
where $d\mu:=d\mu_1 * d\mu_2$ (additive convolution in $\R$). If
$d\mu_1,d\mu_2$ have density functions $\omega_1,\omega_2$
respectively then $d\mu$ has density function
$$
\omega_1*\omega_2(x):=\int_{-\infty}^\infty \omega_1(y)\omega_2(x-y)\
dy.
$$

If $d\mu$ is a measure on $[-\infty,\infty]$ 
 then for any real number $a$ we have
\begin{equation}
\label{mom-shift}
e^{-at}M_\mu(t)=\sum_{k\geq 0} M_k^{(a)} \frac{t^k}{k!},
\end{equation}
where $M_k^{(a)}$ denotes the $k$-th moment of the translated measure
$d\mu(x+a)$.

  The kind of statement we look for is the following. 

  \begin{definition}The sequence of varieties $X_n$ have {\it limiting
      Betti distribution $d\mu$} if there exist real constants
    $\alpha_n,\beta_n,\gamma_n$ with $\alpha_n,\gamma_n >0$ such that
$$
\lim_{n\rightarrow \infty}\frac 1{\gamma_n}\Phi_n(\alpha_n x+\beta_n)
= \Phi(x)
$$
at all points $x$ of continuity of $\Phi$, where $\Phi_n$ and $\Phi$ are the
cumulative density function associated to $E(X_n,q)$ and $d\mu$
respectively. I.e.,
$$
\Phi_n(x) = \int_{-\infty}^x d\mu_n, \qquad \qquad \Phi(x) =
\int_{-\infty}^x d\mu.
$$
where $d\mu_n$ is the measure associated to $E(X_n,q)$.

\end{definition}
This notion of convergence is called {\it convergence in
  distribution}.  Typically a proof of such convergence goes by
proving that the appropriately scaled sequence of moments of
$d\mu_n$ converges to those of $d\mu$ by means of the following
\begin{theorem}
\label{mom-lim-thm}
Suppose that the distribution of $X$ is determined by its moments,
that the $X_n$ have moments of all orders,  and
that $lim_{n\rightarrow \infty} M_k(X_n) = M_k(X)$ for $k = 0, 1, 2,
\ldots$. Then $X_n$ converges in distribution to $X$.
\end{theorem}
\begin{proof}
This is Theorem 30.2 in \cite{billingsley}.
\end{proof}

\subsection{Large Tori and Grassmannians} 
\label{torigras}
When  $X_n=\M^n_1=T^*{\rm Jac}(C_{n})$ it has the topology of a $2n$-dimensional split torus. 
Then
$$
b_i(X_n)=\binom {2n}i .
$$
It follows from the Central Limit theorem that the sequence $X_0, X_1,\ldots$ has Gaussian limiting
Betti distribution. This we could observe in Figure~\ref{TJacC100} for $n=100$.

Now fix a positive integer $r$ and consider  the 
Grassmanian variety $X_n:=G_r^{r+n}$ parametrizing $r$-dimensional
planes in an ambient space of dimension $r+n$. It is well know that
the number of points of the Grassmanian over a finite field is given a
by a $q$-binomial number. Explicitly,
\begin{equation}
\label{grassm-count}
E_n(q):=\#G_r^{r+n}(\F_q)=
\left[\begin{array}{c}
n+r\\
r
\end{array}\right]=\frac{\prod_{j=1}^r(q^{n+j}-1)}
{\prod_{j=1}^r(q^j-1)}.
\end{equation}
Consider the $j$-th factor of this product and assume that $n=jm$ for
some integer $m$. Then 
\begin{equation}
\label{grassm-j-factor}
\frac{q^{n+j}-1}
{q^j-1}=
q^{n/2}\frac{q^{(n+j)/2}-q^{-(n+j)/2}}{q^{j/2}-q^{-j/2}} 
=q^{n/2}\sum_{i=0}^mq^{j(i-m/2)}.
\end{equation}
If we now replace $q$ by $q^{1/n}$ and ignore the power of $q$
prefactor we obtain
$$
\sum_{i=0}^mq^{i/m-1/2}.
$$
As $m$ approaches infinity the associated density function converges
to $\chi^{(1)}:=\chi_{[-1/2,1/2]}$, the characteristic function of the
interval $[-1/2,1/2]$.

Therefore, if $n$ is divisible by all $j=1,2,\ldots,r$ then $E_n(q)$
has associated density function that scaled appropriately approximates
the $r$-th iterated convolution of~$\chi^{(1)}$: 
$$
% \chi^{(r)}:= \chi_{[-1/2,1/2]} * \cdots * \chi_{[-1/2,1/2]}.
\chi^{(r)}:= \underbrace{\chi^{(1)} * \cdots * \chi^{(1)}}_r.
$$
Consequently, we should expect $X_n$ to have limiting Betti
distribution $\chi^{(r)}$. To prove this for the full sequence $X_n$
(and not just for the subsequence of $n$'s in the previous argument) we
consider the moment generating function $E(e^t)$ (see \eqref{mom-gf}).
By \eqref{grassm-j-factor} we have 
$$
e^{-rt/2}E_n(e^{t/n})= 
\prod_{j=1}^r
\frac{e^{\tfrac12(1+j/n)t}-e^{-\tfrac12(1+j/n)t}}{e^{\tfrac12(j/n)
    t}-e^{-\tfrac12(j/n) t}}.
$$
Taking the limit as $n\rightarrow \infty$ we obtain
\begin{equation}
\label{grassm-limit}
\lim_{n\rightarrow \infty} \frac{r!}{n^r}e^{-rt/2}E_n(e^{t/n}) =
\left(\frac{e^t-e^{-t}}t\right)^r.
\end{equation}
The function $(e^{t/2}-e^{-t/2})/t$ is precisely the moment generating
function for $\chi^{(1)}$. Hence~\eqref{grassm-limit} shows that
indeed the $X_n$ have limiting Betti distribution $\chi^{(r)}$
by Theorem~\ref{mom-lim-thm}. 

The density functions $\chi^{(r)}$ have a long history. In
approximation theory they are called {\it central $B$-splines}
(see~\cite{butzer} for details). The support of $\chi^{(r)}$ is
the interval $[-r/2,r/2]$, it is a $C^{n-2}$ function and a polynomial
of degree $n-1$ in each subinterval $[m-r/2,m+1-r/2]$ for
$m=0,1,\ldots,r-1$. By the central limit theorem the distribution
$\chi^{(r)}$ approaches a Gaussian distribution
as~$r\rightarrow~\infty$. More precisely~\cite[(4.7)]{butzer},
$$
\lim_{r\rightarrow \infty}
\sqrt{r/6}\, \chi^{(r)}\left(\sqrt{r/6}x\right)
= \frac 1{\sqrt{\pi }}e^{-x^2}.
$$
\subsection{Large Hilbert schemes of points on $\C^2$}
Consider the sequence of varieties $X^{[n]}$ the Hilbert scheme of $n$
points on $X=\C^2$. It follows from~\eqref{number parts} that if
$d\mu_n$ denotes the discrete measure associated to $X^{[n]}$ then
\begin{equation}
\label{hilb-partitions}
\Phi_n(x)=\int_{-\infty}^x d\mu_n= \#\{\lambda \, |\, |\lambda|=n,
l(\lambda)\leq x\}.
\end{equation}
By a theorem of Erd\H{o}s and Lehner \cite[Theorem 1.1]{erdos-lehner}
\begin{equation}
\label{gumbel}
\lim_{n\rightarrow \infty} \frac1{p(n)}\Phi_n(\alpha_nx+\beta_n) =
e^{-c^{-1}e^{-cx}},
\end{equation}
where $p(n)$ is the total number of partitions of $n$, $c:=\pi/\sqrt6$
and 
$$
\alpha_n:=\sqrt n, \qquad \qquad \beta_n:=2c^{-1}\sqrt n\log n.
$$
It follows from this result that the sequence $X^{[n]}$ has limiting
Betti distribution given by an instance of the {\it Gumbel
  distribution}. These appear as universal distributions when
considering the maximum of samples (rather than the average as in the
central limit theorem). Such {\it extreme value distributions} are
relevant in the prediction of extreme natural fenomena like
earthquakes, floods, etc. In our concrete case the density function is
$\omega(x):=\exp(-c^{-1}e^{-cx}-cx)$ (the derivative of the right hand
side of~\eqref{gumbel}) whose graph is given in
Fig.~\ref{gumbel-distr}. This should be compared, after scaling and
reflection in the $y$-axis, to Figure~\ref{C2_500_}.
 \begin{figure}[H] 
\begin{center}
\includegraphics[height=7cm,width=12cm]{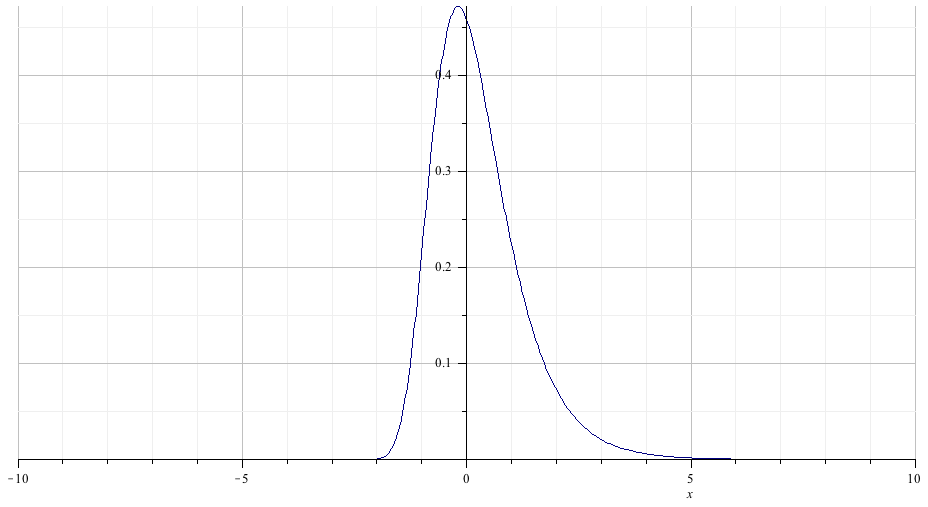}
\caption{Gumbel distribution $\exp(-c^{-1}e^{-cx}-cx)$}
\label{gumbel-distr}
\end{center}
\end{figure}

It is far from clear a priori why such a distribution would appear as
a limiting Betti distribution of the Hilbert scheme of points on
$\C^2$. It would be interesting to see what other limiting Betti
distributions occur for the sequence $S^{[n]}$ for an arbitrary smooth
surface $S$.

On the other hand, it is not hard to convince ourselves of the
relevance of extreme value distributions for our problem 
given~\eqref{hilb-partitions}. Indeed, $l(\lambda)=\lambda'_1$, where
$\lambda'$ is the partition dual to $\lambda$. In other words, the
length of a partition equals the largest part of its dual.

\subsection{Large toric hyperk\"ahler varieties}
\label{largetoric}
Take $X_n=\M^{K_n}_{\bf 1}$ the hyperk\"ahler toric quiver variety associated to the complete graph
$K_n$ on $n$-vertices. As mentioned in \S~\ref{bettihk} we have that $E(X_n,q)$ equals a
polynomial invariant, the external activity polynomial of the graph $K_n$.

We claim that $X_n$ has a limiting Betti distribution known as the
{\it Airy distribution} (for another instance of this phenomenon
see~\cite{reineke2}). This distribution appears in several different
combinatorial and physical problems and has its origin as the
distribution of the area of a Brownian excursion. Its density function
is rather complicated to describe explicitly (its graph is given in
Fig.~\ref{airy-distr}). In particular, its relation to the classical
Airy function, from where the distribution gets its name, is not that
straightfoward to state. We will work instead with the moments which
luckily determine the Airy distribution uniquely
(see~\cite[Thm. 3]{Takacs}). 

There is a sizable literature on the Airy distribution; we will use
the survey~\cite{janson} as our main reference and point to the
interested reader to the works cited there for more details. But we
should warn the reader that there exists a different but related
distribution called in the literature the map-Airy distribution. 

 \begin{figure}[H] 
\begin{center}
\includegraphics[height=7cm,width=12cm]{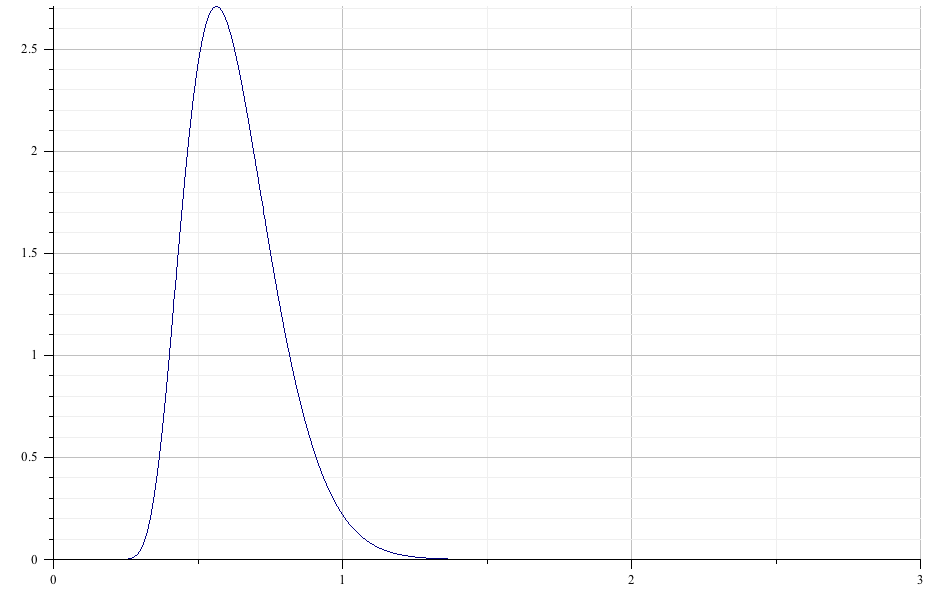}
\caption{Airy distribution density function}
\label{airy-distr}
\end{center}
\end{figure}

We start by defining the rational constants $c_k$ by means of the
following expansion
\begin{equation}
\sum_{k\geq 1} c_k T^k:=\log \sum_{n\geq 0}
\frac{(1/6)_n(5/6)_n}{n!}\left(\frac 3 2T\right)^n,
\end{equation}
where $(a)_n:=a(a+1)\cdots (a+n-1)$ is the Pochhammer symbol.
The first few values are
\begin{center}
\begin{tabular}{c|ccccc}
$k$&1&2&3&4&5\\
\hline
$c_k$&5/24 &5/16 & 1105/1152& 565/128&82825/3072
\end{tabular}
\end{center}
We call $c_k$ the {\it Wright
  constants}. There is a large number of different normalizations of
these constants in the literature; see~\cite{janson} for a
comprehensive comparison between these. It is not hard to show
that we have
$$
c_k=\sum_{Q} \frac 1{\Aut(Q)},
$$
where $Q$ runs over all connected trivalent graphs on $k$
(unlabeled) vertices and $\Aut(Q)$ denotes its group of
automorphisms. 

Now we define the constants $\rho_k$ by~\cite[(34)]{janson}
\begin{equation}
\label{c-defn}
\rho_{-1}:=1,\qquad \rho_0:=\frac{\sqrt{2\pi}}4,\qquad
\rho_k:=\frac{\sqrt\pi}{2^{\tfrac12\left(3k-1\right)}
  \Gamma\left(\tfrac32k\right)}c_k, \qquad  
k\geq 1.
\end{equation}
Then the $k$-th moment $M_k$ of the Airy distribution is
$k!\rho_{k-1}$~\cite[(36)]{janson}; concretely,
$$
M_0=1,\qquad M_1=\frac{\sqrt{2\pi}}4=0.626657068\ldots \qquad
M_2=\frac 5{12}=0.416666666\ldots,\qquad \text{etc.}
$$
(more numerical values are listed in~\cite[Table 4]{Takacs}).

To connect back to the external activity polynomial of $K_n$ note that
by~\eqref{R-pol-defn} 
\begin{equation}
\label{K_n-R-pol}
R_{K_n}(q)=\sum_{k\geq 0}C_{n,n+k-1}(q-1)^k,
\end{equation}
where $C_{n,m}$ denotes the number of connected graph on $n$ labeled
vertices with  $m$ edges. It follows from~\eqref{mom-gf-E} that 
%% $\mu_E$ with $E_n:=R_{K_n}$ we have
$$
m_{\mu_{E_n}}(\eta)=R_{K_n}(1+\eta)=\sum_{k\geq 0}C_{n,n+k-1}\eta^k.
$$
Hence the $k$-th factorial moment $m_{n,k}$ of $\mu_{E_n}$ is
precisely $C_{n,n+k-1}$. 

By a standard result $C_{n,n-1}$, the number of trees on $n$ labeled
vertices, is $n^{n-2}$. Wright proved~\cite[(20)]{janson} that for
fixed $k\geq 0$ we have
\begin{equation}
\label{wright-asympt}
C_{n,n+k-1}\sim \rho_{k-1}n^{n-2+\tfrac32 k}, \qquad n\rightarrow \infty.
\end{equation}
Therefore,
\begin{equation}
\label{wright-asympt-1}
\frac{m_{n,k}}{m_{n,0}} \sim \rho_{k-1}n^{\tfrac32 k}, \qquad
n\rightarrow \infty. 
\end{equation}
Our claim on the limiting Betti distribution of $X_n$ now follows
from Theorem~\ref{mom-lim-thm}.

 Consider now the varieties $X_{m,n}$ associated to the complete
bipartite graph $K_{m,n}$. By~\cite{hausel-sturmfels}
$E_{m,n}(q):=E(X_{m.n},q)$ is the external activity polynomial of $K_{m,n}$. Denote
by $M_k^{m,n}$ de the $k$-th moment of $\mu_{E_{m,n}}$. It is not hard
to prove that for fixed $n$, $M_k^{m,n}/n^m$ is a polynomial in $m$ of
degree $k+n-1$. We normalize the leading coefficient as follows
$$
M_k^{m,n} \sim \alpha_{n,k}  n^{m-n-k+1} m^k,
\qquad m\rightarrow \infty. 
$$
for some constants $\alpha_{n,k}$. A computation shows that
$$
\beta_{n,k}:=\binom{n+k-1}k\alpha_{n,k}=
n\sum_{|\lambda|=n} (-1)^{l(\lambda)-1}
\frac{(l(\lambda)-1)!}
{\prod_{i\geq 1} m_i!i!^{m_i}}
n(\lambda')^{n+k-1},
$$
where the sum is over all partitions $\lambda$ of $n$, 
$m_i:=m_i(\lambda)$ is the multiplicity of $i$ in $\lambda$ and 
$$
n(\lambda'):= \sum_{i\geq 1}\binom i 2 m_i.
$$
Here is a table with the first few values of $\beta_{n,k}$, which are
non-negative integers.
\begin{center}
\begin{tabular}{c|rrrrrr}
$n\backslash k $ & 0 & 1 & 2 & 3 & 4 & 5\\
\hline
1 & 1 & 0 & 0 & 0 & 0 & 0\\
2 & 1 & 1 & 1 & 1 & 1 & 1 \\
3 & 3 & 12 & 39 & 120 & 363 & 1092 \\
4 & 16 & 156 & 1120 & 7260 & 45136 & 275436 \\
5 & 125 & 2360 & 30925 & 353500 & 3795225 & 39474960
\end{tabular}
\end{center}
Note that $\beta_{n,0}=n^{n-2}$.  We have the following generating
function identity
\begin{equation}
\label{bipart-gf}
\sum_{k\geq 0}\alpha_{n,k}\frac{t^k}{k!} =
\left(\frac{e^t-1}t\right)^{n-1}R_{K_n}(e^t),
\end{equation}
where $R_{K_n}$ is the external activity polynomial of the complete graph $K_n$.  It
follows that  $\alpha_{n,k}$ is the $k$-th moment of
$$
\tilde \omega_n:= \underbrace{\chi_{[0,1]}dx * \cdots *
 \chi_{[0,1]}dx}_{n-1} *d\mu_n, 
$$
where $\chi_{[0,1]}$ is the characteristic function of the interval
$[0,1]$ and $d\mu_n$ is the measure $d\mu_{E_n}$ for
$E_n:=E(X_n,q)$ and $X_n$ the variety of 4) associated to $K_n$.  
We deduce that the sequence of varieties $X_{m,n}$ for fixed $n$
has limiting Betti distribution $\tilde\omega_n$ as $m\rightarrow
\infty$. 

As in the case  of the Grassmanian the density function $\tilde
\omega_n$ is a spline of degree $n-2$. For example, for $n=3$ we find
that
$$
\frac{\alpha_{3,k}}{\alpha_{3,0}}=\frac{3^{k+1}-1}{(k+2)(k+1)}
$$
are the moments of the density function
$$
%% \tilde \omega_3(x):=
\begin{cases}
0 & x < 0 \\
\tfrac23x & 0\leq x <  1\\
 -\tfrac13x+1 & 1 \leq x < 3 \\
0 & 3 \leq x
\end{cases}.
$$
This is, up to scaling, a continuous piecewise linear approximation to
the density function of the Airy distribution (its graph is given in
Fig.~\ref{spline}).
 \begin{figure}[H] 
\begin{center}
\includegraphics[height=7cm,width=12cm]{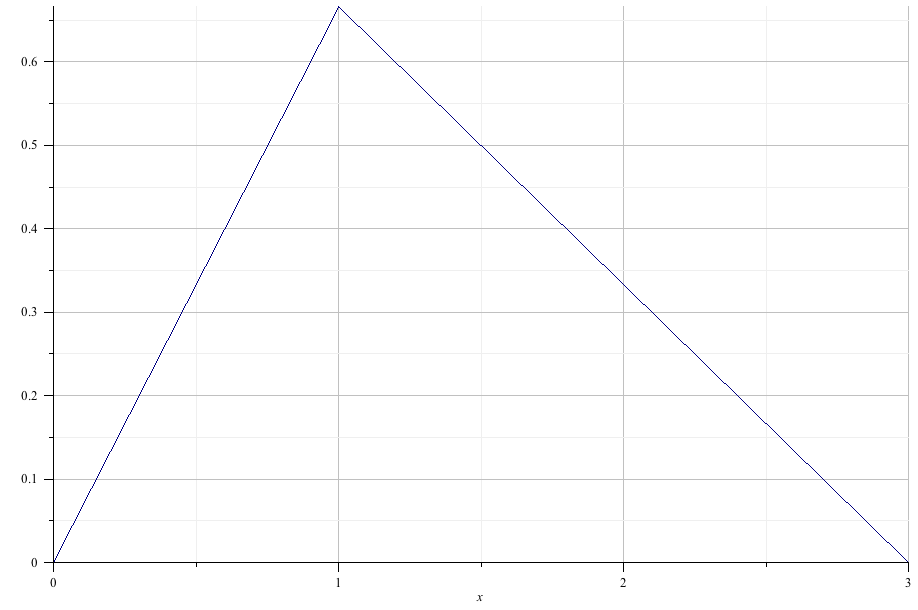}
\caption{Spline approximation to the scaled Airy distribution \label{spline}}
\end{center}
\end{figure}

On the other hand the coefficients of the external activity polynomial for the graph
$K_{3,100}$ are shown in Fig.~\ref{K-3-100}.
 \begin{figure}[H] 
\begin{center}
\includegraphics[height=7cm,width=12cm]{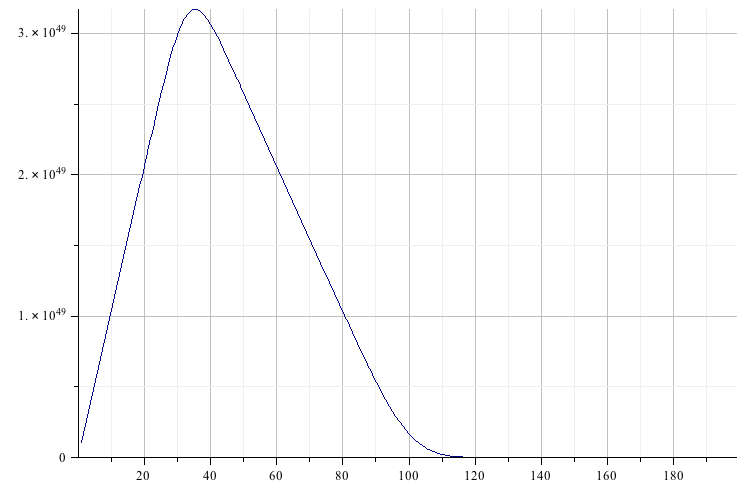}
\caption{External activity polynomial of $K_{3,100}$\label{K-3-100}
}
\end{center}
\end{figure}

The first factor on the right hand side of~\eqref{bipart-gf} has the
expansion
$$
\left(\frac{e^t-1}t\right)^{n-1}=:\sum_{k\geq 0} \gamma_{n,k} t^k, \qquad
\qquad \gamma_{n,k}:=(n-1)!\sum_{k\geq 0}\stirling {n+k-1}{n-1}
\frac{t^k}{(k+n-1)!},
$$
where $\stirling a b$ denotes the Stirling numbers of the second
kind. We have for fixed $k$
$$
\gamma_{n,k} \sim \frac{n^k}{2^kk!}, \qquad
 n\rightarrow \infty.
$$
Hence by replacing $t$ by $t/n^{3/2}$ in~\eqref{bipart-gf} we see
that, appropriately scaled, the distributions $\tilde \omega_n$
converge to the Airy distribution as $n\rightarrow \infty$.  We
actually expect that the double sequence $X_{m,n}$ (rather than the
iterated limit $\lim_n\lim_m$ we considered) should have the Airy
distribution as its limiting Betti distribution.

\subsection{Large quiver varieties: heuristics}
\label{largequiver}

Let $Q=(\calV,\calE)$ be a quiver (see~\ref{nakajima}).  Given a
dimension vector $\vv=(v_1,\ldots,v_n)$, let $A_\vv(q)$ be the Kac
polynomial for $Q,\vv$.

Recall~\eqref{cb} that for $\vv$ indivisible $A_\vv(q)$ is the reverse
of the Poincar\'e polynomial of a certain smooth quiver
variety~$\M_\vv$. We may consider the collection of all such Nakajima
quiver varieties associated to indivisible dimension vectors~$\vv$ for
a fixed quiver $Q$.  We expect that when~$\vv$ tends to infinity
generically the corresponding varieties $\M_\v$ have the Airy
distribution as limiting Betti distribution independently of the
quiver $Q$.

In this section we present some heuristics in support of this expected
{\it universality} property of the Airy distribution. Concretely, we
have Hua's formula~\eqref{hua} for $A_\v(q)$, which though somewhat
unwieldy, it has a structure quite similar to~\eqref{complete} for
computing the external activity polynomial of all complete graphs. We
have already shown in~\S~\ref{largetoric} that in this case we have
the Airy distribution as the limit. The key fact is the asymptotic
calculation of the moments~\eqref{wright-asympt-1}, which in turn
boils down to~\eqref{wright-asympt}.

Flajolet and his collaborators~\cite{Flajolet-et-al} use a saddle
point analysis to prove~\eqref{wright-asympt}. In a future publication
we hope to apply to~\eqref{hua} this saddle point analysis to prove
the expected universality of the Airy distribution for quiver
varieties.  We outline below the key ingredients of the saddle point
approach following closely~\cite{Flajolet-et-al}, to which we refer
the reader for details, and end with a brief discussion on how it
could be applied to the case of quiver varieties.

Using the standard Gaussian integral
$$
e^{-\tfrac12y^2}=\frac1{\sqrt{2\pi}}\int_{-\infty}^\infty
e^{ixy-\tfrac12x^2}\,dx
$$
we rewrite the right hand side of~\eqref{complete} with

$q=1+\eta=e^{-t}$ and $t>0$ as
$$
F(\eta,T):=\frac1{\sqrt{2\pi t}}\int_{-\infty}^\infty 
e^{-\tfrac12t^{-1}x^2} \sum_{n\geq 0}
e^{\left(ix+\tfrac12t\right)n}\frac{(T/\eta)^n}{n!} \, dx
= \frac1{\sqrt{2\pi t}}\int_{-\infty}^\infty 
e^{\phi(x,\eta,T)}\,dx,
$$
where
$$
\phi(x,\eta,T):=-\tfrac12t^{-1}x^2+e^{ix+\tfrac12t}(T/\eta).
$$

Note that by~\eqref{complete} and~\eqref{K_n-R-pol}
\begin{equation}
\label{logF-expansion}
\eta\log F(\eta,T)=C_0(T)+C_1(T)\eta +\cdots,
\end{equation}
where 
$$
C_k(T):=\sum_{n\geq 0} C_{n,n+k-1}\frac{T^n}{n!}
$$
is the exponential generating function for connected graphs on $n$
vertices with a fixed $n+k-1$ first Betti number.

In order to study the asymptotics of $C_{n,n+k-1}$ and
prove~\eqref{wright-asympt}, we compute the asymptotic behavior of
$F(\eta,T)$ as $\eta$ approaches zero by using the saddle point
method. For this purpose we compute the critical points of
$\phi_{-1}(x)$, where
$$
\phi(x,\eta,T)=\phi_{-1}(x,T)\eta^{-1}+\phi_0(x,T)+\phi_1(x,T)\eta +\cdots.
$$
We have
$$
\phi_{-1}(x,T)=\tfrac12 x^2 +e^{ix}T
$$
and the critical points are the solutions $x=x(T)$ of the equation
$e^{ix}T=ix.$ In a neighborhood of $T=0$ we can solve this equation
with $ix=w=w(T)$ where
$$
w(T)=\sum_{n\geq 1} n^{n-1}\frac{T^n}{n!},
$$
is the {\it tree function} satisfying what we call the~{\it saddle
  point equation}
\begin{equation}
\label{tree-eqn}
Te^w=w.
\end{equation}
In terms of the {\it saddle point parameter} $w$ the expansion
coefficients $C_k(T)$ of~\eqref{logF-expansion} have the following
form
\begin{eqnarray}
\label{logF-expansion-1}
 C_0=w-\tfrac12w^2,\\
 C_1=-\tfrac12\log(1-w)+\tfrac12w +\tfrac14w^2\\
C_k = \frac{E_{k-1}(w)}{(1-w)^{3(k-1)}}, \qquad k >1
\end{eqnarray}
for certain polynomials $E_k$. It follows from~\cite[Lemma
2]{Knuth-Pittel} that the 
asymptotics of $C_{n,n+k-1}$ for large $n$ and $k>1$ fixed
is of the form
\begin{equation}
\label{wright-asympt-2}
C_{n,n+k-1}\sim
E_{k-1}(1)\frac{\sqrt{2\pi}\ n^{n+\tfrac32k-2}}{2^{\tfrac32(k-1)}\
  \Gamma(\tfrac32(k-1))}, 
\end{equation}
if $E_{k-1}(1)$ does not vanish. As we will now see, in fact,
\begin{equation}
\label{E-value}
E_k(1)=c_k,
\end{equation}
 the Wright constants defined in~\eqref{c-defn}. Equivalently, as $w$
 approaches $1$
$$
C_k\sim c_k(1-w)^{-3(k-1)}+\cdots, \qquad k>1.
$$
Hence,~\eqref{wright-asympt-2} is the same as~\eqref{wright-asympt}.

We make the change of variables $x=y-iw$ in the integral and get
$$
\phi_{-1}(x,T)=-\tfrac12w^2+w + \psi(y,w),
$$
where $ \psi(y,w)=\tfrac12(1-w) y^2
+(e^{iy}-1-iy+\tfrac12y^2)w$. Therefore,
\begin{equation}
\label{F-integral}
F(\eta,T)=e^{\eta^{-1}\left(-\tfrac12w^2+w\right)}
\frac1{\sqrt{2\pi t}}\int_\calC 
e^{\eta^{-1}\psi(y,w)}h(y,\eta,w)\ dy
\end{equation}
for a certain funtion $h(y,\eta,w)$ holomorphic in $\eta$ and an
appropriate contour $\calC$.

To prove~\eqref{E-value} and get ~\eqref{wright-asympt-2} one
studies the behavior of $C_k$ as $w$ approaches $1$. Note however that
the critical point $y=0$ of
$$
\psi(y,w)=\tfrac12(1-w) y^2 -\tfrac i6wy^3+O(y^4)
$$
becomes degenerate for $w=1$. It is well known that in this situation
the standard saddle point method has to be modified to incorporate
higher order terms. We have here the simplest case of this phenomenon
known as the~{\it coalescing of saddle points}~\cite{Chester-et-al}.

We homogenize by letting  $y_*,w_*,\eta_*$ be new variables defined by
$$
w=1-w_*,\qquad y=w_*y_*,\qquad \eta=w_*^3\eta_*.
$$
Then
$$
\eta^{-1}\psi(y,w)=\eta_*^{-1}y^2_*(1-\tfrac i3y_*+w_*y_*r(y_*,w_*))
$$
for some power series $r(y_*,w_*)$ in $y_*$ with coefficients
polynomial in $w_*$. Now we can let $w_*$ approach zero and check that
up to a tractable factor the integral in~\eqref{F-integral} can be
replaced by
$$
\frac1{\sqrt{2\pi \eta_*}}\int_{-\infty}^\infty
  e^{\tfrac12\eta_*^{-1}\left(y_*^2-\tfrac i3y_*^3\right)}\ dy_*. 
$$
Since the asymptotic expansion of this integral is precisely 
$$
\sum_{n\geq 0}  \frac{(1/6)_n(5/6)_n}{n!}\left(\frac 3 2\eta_*\right)^n
$$
\eqref{E-value} follows.

To summarize, we prove that we have Betti limiting distribution the
Airy distribution by computing the limiting moments. To do this, we

\begin{enumerate}
\item
 Express the generating function of all moments as an
  integral.
\item
Find the critical points of the dominant exponential factor of the
integrand, which satisfy saddle point equations.

\item Show that in terms of the saddle point parameters the generating
  function $C_k$ of the $k$-th moment (for $k>1$) becomes a rational
  function whose leading term involves the Wright constants $c_k$.
\item
Deduce that the limiting moments are those of the Airy distribution.
\end{enumerate}

As mentioned, we expect that these same steps can be applied to Hua's
formula~\eqref{hua} to study the limiting distribution of quiver
varieties for a fixed but arbitrary quiver as mentioned earlier. The
first two steps are fairly routine. In general the saddle point
equations however will determine an algebraic variety of higher
dimension (equal to the number of nodes in the quiver). Carrying
through the last two steps then becomes more of a challenge but at
least the generic behaviour, when the dimension vector components
increase to infinity independently, should be as above.  We will
revisit this issue in a later publication.

\section{Results and speculations on the asymptotics of discrete
  distributions} 

In the previous section we proved and gave heuristics for some
asymptotical results on the distribution of Betti numbers of certain
families of semiprojective hyperk\"ahler varieties. Not surprisingly,
we found in \S\ref{torigras} that the Gaussian distribution appears in
several examples. The classical binomial distribution, given for us as
the Betti numbers of tori, is the most well known example of such
asymptotic behavior. We also found in \S\ref{torigras} that the Betti
numbers of certain families of Grassmannians also have Gaussian
limiting distributions. The same behavior was already studied by
Tak\'acs \cite{takacs} in his studies of coefficients of q-binomial
coefficients. More recent work of Stanley and Zanello
\cite{stanley-zanello} gives new results on asymptotics of these
quantities, as well as studies unimodality properties of various
sequences; not unlike our sequences of (even) Betti numbers of
semiprojective hyperk\"ahler varieties.

In fact, in our computed examples we observed that the sequence of even (similarly odd) Betti numbers form a unimodal sequence. This result follows from the Hard Lefschetz theorem for a smooth projective variety; or a semiprojective variety with core which is smooth projective. However it is unclear why this property may hold in larger generality. Clearly the even Betti numbers of smooth affine varieties will not necessarily be unimodal, as the case of $\SL(n,\C)$ shows.  In fact starting with Stanley's studies \cite{stanley2} several combinatorial sequences have been conjectured and 
some proved to be unimodal. In particular recently Huh in \cite{huh} proved that the $h$-vector of a representable matroid is log-concave and thus unimodal, in particular proving a long-standing conjecture of Colbourn on log-concavity of the external activity polynomial. 
This proves that the Betti numbers of 
toric hyperk\"ahler varieties (which are the h-vectors of rationally representable matroids) 
form a unimodal sequence. 

However, some combinatorial counterexamples are relevant for us
too. For example, in our geometric language, Stanton \cite{stanton}
found examples of Poincar\'e polynomials of closures of Schubert cells
in Grassmannians, which are not unimodal. It could be relevant for us
as the closure of the Schubert cell is an equidimensional, in fact
irreducible, proper variety with a paving, and thus has pure
cohomology. Thus potentially it could be the core of a semiprojective
hyperk\"ahler variety.

We also mention two recent appearances of the Gaussian distribution as
a limit of series of discrete distributions. First in \cite{ein-etal}
it is conjectured that sequences of ranks of certain syzygies of a
smooth projective varieties also have Gaussian distribution in the
limit.  More directly relevant for us is the recent work of Morrison
\cite{morrison}. It was proved there that the sequence of discrete
distributions given by the motivic DT invariants of $(\C^3)^{[n]}$ is
also normally distributed in the limit. In fact the generating
function \cite{behrend-etal} of such motivic DT-invariants is
similar, at least in the limit $m\to \infty$, to generating functions
of the twisted ADHM spaces $\M_{n,m}$ we discussed in
\S\ref{adhms}. Thus it is conceivable, that in an appropriate limit
the Betti numbers of $\M_{n,m}$ will also be
distributed normally; Figure~\ref{M(40,20)} seems to support this
possibility.

It is worthwhile noticing that the graphs of the Gumbel
(Figure~\ref{gumbel-distr}) and Airy (Figure~\ref{airy-distr})
distributions seem very similar to the naked eye. In fact, they are
really different (for example, they have different decay rate at the
tails). However, looking at the distribution of Betti numbers in the
case of the toric quiver variety $\M^{K_{40}}_{\bf 1}$
(Figure~\ref{toricK40}) and the Hilbert scheme $(\C^2)^{[500]}$ 
(Figure~\ref{C2_500_}) one might easily believe they are
approaching a common limit.

The sequence of graphs, such as the complete graphs $K_n$, we studied
in \S\ref{largetoric} is convergent in the sense of Lov\'asz-Szegedy
\cite{lovasz-szegedy}.  The continuous limit for their extremal
activity polynomial we found there could possibly be related to some
invariants of the limiting objects. This also raises the possibility
of existence of a limiting object to our sequences of hyperk\"ahler
manifolds, whose ``Poincar\'e series" in the appropriate sense would
agree with our limiting distribution.

The Airy distribution in \S\ref{largetoric}, governing the limit of
Betti numbers of the toric quiver varieties attached to complete
graphs and possibly sequences of more general quiver varieties as
discussed in \S\ref{largequiver} was earlier noticed to be the
limiting distribution of Betti numbers of certain non-commutative
Hilbert schemes by Reineke in \cite[Theorem 6.2]{reineke2}. In fact,
our heuristics in \S\ref{largequiver} were motivated by an effort to
systematize a proof of such results.

Intrestingly, very similar analysis to the saddle point method in \S\ref{largequiver}  have been  
used in \cite{dimofte-etal} to study the asymptotic properties of coloured ${\rm SU}(2)$ Jones polynomials. Also the large $N$ or t'Hooft limit of various $U(N)$ gauge theories, studied extensively by string theorists, also involves asymptotic studies not unlike ours. Maybe in these contexts we will find an explanation of the continuous looking limit distribution (Figure~\ref{hqtn8g2}) of the Betti numbers 
of the moduli space of rank $n$ Higgs bundles $\M_n^g$ as $n\to \infty$.

\end{document}